\documentclass[12pt]{amsart}

\usepackage{amsmath, amssymb}
\usepackage[all]{xy}

\theoremstyle{plain}
\newtheorem{theorem}{Theorem}[section]
\newtheorem{prop}[theorem]{Proposition}

\theoremstyle{remark}
\newtheorem{remark}[theorem]{Remark}

\let\a\alpha
\let\b\beta
\let\g\gamma
\let\d\delta
\let\e\varepsilon
\let\z\zeta

\let\l\lambda
\let\m\mu

\let\f\varphi
\let\L\Lambda

\let\Om\Omega

\def\lra{\longrightarrow}
\def\egal{\ar@{=}}

\def\A{\mathcal A}

\def\C{\mathbb C}
\def\CC{\mathcal C}
\def\E{\mathcal E}
\def\F{\mathcal F}
\def\G{\mathcal G}

\def\I{\mathcal I}
\def\J{\mathcal J}

\def\P{\mathbb P}

\def\NN{\mathcal N}
\def\O{\mathcal O}
\def\TT{\mathcal T}
\def\U{\mathbb U}

\def\W{\mathbb W}

\def\Ker{{\mathcal Ker}}
\def\Coker{{\mathcal Coker}}
\def\Im{{\mathcal Im}}

\def\D{{\scriptscriptstyle \operatorname{D}}}
\def\T{{\scriptscriptstyle \operatorname{T}}}
\def\ss{{\scriptstyle \operatorname{ss}}}
\def\st{{\scriptstyle \operatorname{s}}}
\def\dd{\operatorname{d}}
\def\EE{\operatorname{E}}
\def\H{\operatorname{H}}
\def\h{\operatorname{h}}
\def\M{\operatorname{M}_{{\mathbb P}^2}}
\def\N{\operatorname{N}}
\def\pp{\operatorname{p}}
\def\PP{\operatorname{P}}
\def\SS{\operatorname{S}}
\def\TTT{\operatorname{T}}

\def\Hom{\operatorname{Hom}}
\def\Aut{\operatorname{Aut}}
\def\Ext{\operatorname{Ext}}
\def\GL{\operatorname{GL}}
\def\length{\operatorname{length}}
\def\rank{\operatorname{rank}}
\def\spann{\operatorname{span}}

\def\Hilb{\operatorname{Hilb}_{{\mathbb P}^2}}

\def\tensor{\otimes}
\def\isom{\simeq}

\newcommand{\noi}{\noindent}

\def\ds{\displaystyle}
\def\ba{\begin{array}}
\def\ea{\end{array}}

\begin{document}


\title[semi-stable plane sheaves with Hilbert polynomial $\operatorname{P}(m)=6m+3$]
{on the geometry of the moduli space of semi-stable plane sheaves
with Hilbert polynomial $\mathbf{\operatorname{\mathbf{P}}(m)=6m+3}$}

\author{mario maican}

\address{Mario Maican \\
Institute of Mathematics of the Romanian Academy \\
Calea Grivi\c tei 21 \\
010702 Bucharest \\
Romania}

\email{mario.maican@imar.ro}

\begin{abstract}
We classify all Gieseker semi-stable sheaves on the complex projective plane
that have dimension $1$, multiplicity $6$ and Euler characteristic $3$.
We show that their moduli space is birational to the blow-up at a special point of a certain
moduli space of semi-stable Kronecker modules.
\end{abstract}

\maketitle

\tableofcontents

\noi
{\sc Acknowledgements.} The author was supported by the
Consiliul Na\c tional al Cercet\u arii \c Stiin\c tifice,
an agency of the Romanian Government,
grant PN II--RU 169/2010 PD--219.


\section{Introduction}

\noi
This paper is a continuation of \cite{mult_six_two}, \cite{mult_six_one}, \cite{mult_five}
and \cite{drezet-maican}.
We are concerned with Gieseker semi-stable sheaves on $\P^2 = \P^2(\C)$ having Hilbert
polynomial $\PP(m)=6m+3$ and with their moduli space $\M(6,3)$,
which is an irreducible projective variety of dimension $37$.
We classify all such sheaves using extensions or locally free resolutions of length $1$.
The third column of the table below lists all such sheaves.
The conditions on the morphisms $\f$ define locally closed subsets $W_i$ inside the
vector spaces of homomorphisms of locally free sheaves
and $X_i$ is the image of $W_i$ in $\M(6,3)$ under the map sending $\f$ to the
stable equivalence class of $\Coker(\f)$.
For a better motivation and for a brief history of the problem we refer to the introductory
sections of \cite{drezet-maican} and \cite{mult_five}.

The sets from the first column of the table are locally closed and cover $\M(6,3)$.
We call them \emph{strata}.
They are defined by the cohomological conditions given in column two
and have the codimension prescribed in column four of the table.
The open stratum $X_0 = \M(6,3) \setminus \overline{X}_1$ is isomorphic to an open subset
of the Kronecker moduli space $\N(6,3,3)$ of semi-stable $3 \times 3$-matrices with
entries homogeneous quadratic forms in three variables.
Let $X_{10}$ be the open subset of $X_1$ defined in section 3.
The union $X=X_0 \cup X_{10}$ is an open subset of $\M(6,3)$ and is isomorphic to
an open subset of the blow-up of $\N(6,3,3)$ at the special point represented by the stable matrix
\[
\left[
\ba{c}
X \\ Y \\ Z
\ea
\right] \left[
\ba{ccc}
X & Y & Z
\ea
\right].
\]
Under this isomorphism $X_{10}$ can be identified with an open subset of the exceptional
divisor of the blow-up.
The complement of $X$ has two irreducible components, each of codimension $2$.
The stratum $X_2$ is an open subset of a fibre bundle over $\N(3,3,2) \times \N(3,2,3)$
with fibre $\P^{21}$.
The open subset of $X_3$ given by stable sheaves is isomorphic to an open subset
of a fibre bundle over $\N(3,3,4)$ with fibre $\P^{21}$.
The stratum $X_4$ is an open subset of a tower of bundles with fibre $\P^{21}$
and base a fibre bundle over $\P^5$ with fibre $\P^6$.
The stratum $X_5$ is isomorphic to an open subset of a fibre bundle over
$\Hilb(2) \times \Hilb(2)$ with fibre $\P^{23}$, where $\Hilb(2)$ denotes the Hilbert scheme
of two points in $\P^2$.
We can partially compactify $X_4$ by adding the scheme $\Hilb(2) \times \Hilb(2)$
as a smooth boundary. The union $X_4 \cup X_5$ is isomorphic to an open subset
of the blow-up of this partial compactification along the boundary.
Under this isomorphism $X_5$ is mapped to an open subset of the exceptional
divisor of the blow-up.
The stratum $X_6$ is an open subset of a fibre bundle over $\P^2 \times \P^2$
with fibre $\P^{25}$.
The stratum $X_7$ is closed and consists of all sheaves of the form $\O_C(2)$
for $C \subset \P^2$ a sextic curve. Thus $X_7 \isom \P^{27}$.
The strata $X_0$, $X_1$, $X_3$, $X_3^\D$ contain points given by properly semi-stable
sheaves. The other strata contain only stable sheaves.
The maps $W_i \to X_i$ are geometric quotients away from properly semi-stable points.
The map $W_0 \to X_0$ is a good quotient.

According to \cite{maican-duality}, the duality map
\[
[\F] \longmapsto [{\mathcal Ext}^1(\F, \omega_{\P^2}) \tensor \O_{\P^2}(1)]
\]
defines an automorphism of $\M(6,3)$.
The strata $X_3$ and $X_3^\D$ are isomorphic under the duality automorphism.
All other strata are preserved by the duality automorphism.

Let $C \subset \P^2$ denote an arbitrary smooth sextic curve and let $P_i$ denote
distinct points on $C$.
One of the irreducible components of the complement of $X$ is the closure of the set of sheaves
of the form
\[
\O_C(3)(-P_1-\cdots -P_7 + P_8),
\]
where $P_1, \ldots, P_7$ are not contained
in a conic curve and no four points among them are colinear.
The other component is the dual of the first component, so it is the closure of the set
of sheaves of the form
\[
\O_C(1)(P_1+ \cdots + P_7 - P_8),
\]
where $P_1, \ldots, P_7$ satisfy the same conditions as above.
The generic sheaves in $X_2$ have the form
\[
\O_C(2)(P_1+P_2+P_3-P_4-P_5-P_6),
\]
where $P_1, P_2, P_3$ are non-colinear, $P_4, P_5, P_6$ are non-colinear.
The generic sheaves in $X_3$ and $X_3^\D$ have the form
\[
\O_C(3)(-P_1 - \cdots - P_6), \qquad \text{respectively} \qquad \O_C(1)(P_1 + \cdots + P_6),
\]
where $P_1, \ldots, P_6$ are not contained in a conic curve.
The generic sheaves in $X_4$ are of the form
\[
\O_C(3)(-P_1 - \cdots - P_6),
\]
where $P_1, \ldots, P_6$ lie on a conic curve and no four of them are colinear.
The generic sheaves in $X_5$, respectively $X_6$, have the form
\[
\O_C(2)(P_1+P_2-P_3-P_4), \qquad \text{respectively} \qquad \O_C(2)(P_1-P_2).
\]

\noi \\
{\sc Notations.}
\begin{align*}
\M(r,\chi) =
& \text{ moduli space of Gieseker semi-stable sheaves on $\P^2$} \\
& \text{ with Hilbert polynomial $\PP(m)=rm+\chi$,} \\
\N(m,p,q) =
& \text{ Kronecker moduli space of semi-stable $q \times p$-matrices} \\
& \text{ with entries in a fixed $m$-dimensional vector space over $\C$,} \\
\Hilb(m) =
& \text{ the Hilbert scheme of $m$ points in $\P^2$,} \\
V =
& \text{ a fixed vector space of dimension $3$ over $\C$,} \\
\P^2 =
& \text{ the projective space of lines in $V$,} \\
\{ X, Y, Z \} =
& \text{ basis of $V^*$,} \\
[\F] =
& \text{ the stable-equivalence class of a sheaf $\F$,} \\
\F^\D =
& \text{ ${\mathcal Ext}^1(\F, \omega_{\P^2})$ if $\F$ is a one-dimensional sheaf on $\P^2$,} \\
X^\D =
& \text{ the image of a set $X \subset \M(r,\chi)$ under the duality automorphism,} \\
X^\st =
& \text{ the open subset of points given by stable sheaves inside a set $X$,} \\
\pp(\F) =
& \text{ $\chi/r$, the slope of a sheaf $\F$ having Hilbert polynomial $\PP(m)=rm+\chi$.}
\end{align*}
For any other unexplained notations and conventions we refer to \cite{mult_five} and
especially to the section of preliminaries of \cite{drezet-maican}.

\begin{table}[!hpt]{}
\begin{center}
{\small
\begin{tabular}{|c|c|c|c|}
\hline \hline
{\tiny stratum}
&
\begin{tabular}{c}
{\tiny cohomological} \\
{\tiny conditions}
\end{tabular}
&
$W$
&
{\tiny codim.}
\\
\hline
$X_0$
&
\begin{tabular}{r}
$\h^0(\F(-1))=0$ \\
$\h^1(\F)=0$\\
$\h^0(\F \tensor \Om^1(1))=0$
\end{tabular}
&
\begin{tabular}{c}
{} \\
$0 \lra 3\O(-2) \stackrel{\f}{\lra} 3\O \lra \F \lra 0$ \\
{}
\end{tabular}
&
$0$
\\
\hline
$X_1$
&
\begin{tabular}{r}
$\h^0(\F(-1))=0$ \\
$\h^1(\F)=0$\\
$\h^0(\F \tensor \Om^1(1))=1$
\end{tabular}
&
\begin{tabular}{c}
{} \\
$0 \lra 3\O(-2) \oplus \O(-1) \stackrel{\f}{\lra} \O(-1) \oplus 3\O \lra \F \lra 0$ \\
$\f_{12}=0$ \\
$\f$ is not equivalent to a morphism of any of the forms \\
${\ds
\left[
\ba{cccc}
\star & 0 & 0 & 0 \\
\star & \star & \star & \star \\
\star & \star & \star & \star \\
\star & \star & \star & \star
\ea
\right], \quad \left[
\ba{cccc}
\star & \star & 0 & 0 \\
\star & \star & 0 & 0 \\
\star & \star & \star & \star \\
\star & \star & \star & \star
\ea
\right], \quad \left[
\ba{cccc}
\star & \star & \star & 0 \\
\star & \star & \star & 0 \\
\star & \star & \star & 0 \\
\star & \star & \star & \star
\ea
\right]
}$ \\
{}
\end{tabular}
&
$1$
\\
\hline
$X_2$
&
\begin{tabular}{r}
$\h^0(\F(-1))=0$ \\
$\h^1(\F)=0$\\
$\h^0(\F \tensor \Om^1(1))=2$
\end{tabular}
&
\begin{tabular}{c}
{} \\
$0 \lra 3\O(-2) \oplus 2\O(-1) \stackrel{\f}{\lra} 2\O(-1) \oplus 3\O \lra \F \lra 0$ \\
$\f_{12}=0$ \\
$\f_{11}$ and $\f_{22}$ are semi-stable as Kronecker modules \\
{}
\end{tabular}
&
$4$
\\
\hline
$X_3$
&
\begin{tabular}{r}
$\h^0(\F(-1))=0$ \\
$\h^1(\F)=1$\\
$\h^0(\F \tensor \Om^1(1))=3$
\end{tabular}
&
\begin{tabular}{c}
{} \\
$0 \lra \O(-3) \oplus 3\O(-1) \stackrel{\f}{\lra} 4\O \lra \F \lra 0$ \\
$\f_{12}$ is semi-stable as a Kronecker module \\
{}
\end{tabular}
&
$4$
\\
\hline
${\ds X_3^\D}$
&
\begin{tabular}{r}
$\h^0(\F(-1))=1$ \\
$\h^1(\F)=0$\\
$\h^0(\F \tensor \Om^1(1))=3$
\end{tabular}
&
\begin{tabular}{c}
{} \\
$0 \lra 4\O(-2) \stackrel{\f}{\lra} 3\O(-1) \oplus \O(1) \lra \F \lra 0$ \\
$\f_{11}$ is semi-stable as a Kronecker module \\
{}
\end{tabular}
&
$4$
\\
\hline
$X_4$
&
\begin{tabular}{r}
$\h^0(\F(-1))=1$ \\
$\h^1(\F)=1$\\
$\h^0(\F \tensor \Om^1(1))=3$
\end{tabular}
&
\begin{tabular}{c}
{} \\
$0 \lra \O(-3) \oplus \O(-2) \stackrel{\f}{\lra} \O \oplus \O(1) \lra \F \lra 0$ \\
$\f_{12} \neq 0$, $\f_{12} \nmid \f_{11}$, $\f_{12} \nmid \f_{22}$ \\
{}
\end{tabular}
&
$5$
\\
\hline
$X_5$
&
\begin{tabular}{r}
$\h^0(\F(-1))=1$ \\
$\h^1(\F)=1$ \\
$\h^0(\F \tensor \Om^1(1))=4$
\end{tabular}
&
\begin{tabular}{c}
{} \\
$0 \to \O(-3) \oplus \O(-2) \oplus \O(-1) \stackrel{\f}{\to} \O(-1) \oplus \O \oplus \O(1) \to \F \to 0$ \\
$\f_{13}=0$, $\f_{12} \neq 0$, $\f_{23} \neq 0$, $\f_{12} \nmid \f_{11}$, $\f_{23} \nmid \f_{33}$ \\
{}
\end{tabular}
&
$6$
\\
\hline
$X_6$
&
\begin{tabular}{r}
$\h^0(\F(-1))=2$ \\
$\h^1(\F)=2$ \\
$\h^0(\F \tensor \Om^1(1))=6$
\end{tabular}
&
\begin{tabular}{c}
{} \\
$0 \lra 2\O(-3) \oplus \O \stackrel{\f}{\lra} \O(-2) \oplus 2\O(1) \lra \F \lra 0$ \\
$\f_{11}$ has linearly independent entries \\
$\f_{22}$ has linearly independent entries \\
{}
\end{tabular}
&
$8$
\\
\hline
$X_7$
&
\begin{tabular}{r}
$\h^0(\F(-1))=3$ \\
$\h^1(\F)=3$ \\
$\h^0(\F \tensor \Om^1(1))=8$
\end{tabular}
&
\begin{tabular}{c}
{} \\
$0 \lra \O(-4) \stackrel{\f}{\lra} \O(2) \lra \F \lra 0$ \\
{}
\end{tabular}
&
$10$
\\
\hline \hline
\end{tabular}
}
\end{center}
\end{table}


\section{The open stratum}

\noi
From \cite{maican}, 4.6, 4.7 and 4.8, we extract the following:

\begin{prop}
\label{2.1}
Let $\F$ be a sheaf giving a point in $\M(6,3)$ and satisfying the conditions
$\h^0(\F(-1))=0$, $\h^1(\F)=0$. Then $\h^0(\F \tensor \Om^1(1))=0$, $1$ or $2$.
The sheaves in the first case are precisely the sheaves with resolution of the form
\[
\tag{i}
0 \lra 3\O(-2) \lra 3\O \lra \F \lra 0.
\]
The sheaves in the second case are precisely the sheaves having resolution of the form
\[
\tag{ii}
0 \lra 3\O(-2) \oplus \O(-1) \stackrel{\f}{\lra} \O(-1) \oplus 3\O \lra \F \lra 0,
\]
such that $\f_{12}=0$, the entries of $\f_{11}$ span a subspace of $V^*$ of dimension
at least $2$, ditto for the entries of $\f_{22}$ and, moreover, $\f$ is not equivalent, modulo the
action of the natural group of automorphisms, to a morphism represented by a matrix of the form
\[
\left[
\ba{cccc}
\star & \star & 0 & 0 \\
\star & \star & 0 & 0 \\
\star & \star & \star & \star \\
\star & \star & \star & \star
\ea
\right].
\]
The sheaves in the third case are precisely the sheaves having a resolution of the form
\[
\tag{iii}
0 \lra 3\O(-2) \oplus 2\O(-1) \stackrel{\f}{\lra} 2\O(-1) \oplus 3\O \lra \F \lra 0,
\]
where $\f_{12}=0$, $\f_{11}$ has linearly independent maximal minors and ditto for $\f_{22}$.
\end{prop}

\noi
Let $\W= \Hom(3\O(-2) \oplus 2\O(-1), 2\O(-1) \oplus 3\O)$ and let $W \subset \W$
be the open subset of injective morphisms with semi-stable cokernel.
Let $W^\st \subset W$ be the open subset of morphisms with stable cokernel.
The group
\[
G = (\Aut(3\O(-2) \oplus 2\O(-1)) \times \Aut(2\O(-1) \oplus 3\O))/\C^*
\]
acts on $\W$ by conjugation. Let $X, X^\st \subset \M(6,3)$ be the images of $W, W^\st$
under the canonical map $\f \mapsto [\Coker(\f)]$.
Note that $X$ is the open dense subset given by the conditions $\h^0(\F(-1))=0$, $\h^1(\F)=0$.

\begin{prop}
\label{2.2}
The canonical map $\rho \colon W \to X$, $\rho(\f)= [\Coker(\f)]$, is a categorical quotient map
for the action of $G$ on $W$.
The restricted map $W^\st \to X^\st$ is a geometric quotient.
\end{prop}

\begin{proof}
As at 4.2.1 \cite{drezet-maican}, we can easily show that $\rho(\f_1) = \rho(\f_2)$ if and only if
$\overline{G \f_1} \cap \overline{G \f_2} \neq \emptyset$.
As at 3.1.6 op.cit., we reduce the problem to the problem of constructing $\f$ in the fibre of $[\F]$
by starting with the Beilinson spectral sequence. We will use the Beilinson spectral sequence I
that converges to $\F(1)$. Its $\EE^1$-term has display diagram
\[
\xymatrix
{
\H^1(\F(-1)) \tensor \O(-1) \ar[r]^-{\f_1} & \H^1(\F) \tensor \Om^1(1) \ar[r]^-{\f_2} & \H^1(\F(1)) \tensor \O \\
\H^0(\F(-1)) \tensor \O(-1) \ar[r]^-{\f_3} & \H^0(\F) \tensor \Om^1(1) \ar[r]^-{\f_4} & \H^0(\F(1)) \tensor \O
}.
\]
By hypothesis some of the above cohomology groups vanish, so the display diagram takes the form
\[
\xymatrix
{
3\O(-1) & 0 & 0 \\
0 & 3\Om^1(1) \ar[r]^-{\f_4} & 9\O
}.
\]
The second term of the spectral sequence has display diagram
\[
\xymatrix
{
3\O(-1) \ar[rrd]^-{\f_5} & 0 & 0 \\
0 & \Ker(\f_4) & \Coker(\f_4)
}.
\]
Thus $\EE^3 = \EE^{\infty}$ hence $\F(1) \isom \Coker(\f_5)$ and $\f_4$, $\f_5$ are injective.
Clearly $\f_5$ factors through $9\O$, so we arrive at a resolution
\[
0 \lra 3\O(-1) \oplus 3\Om^1(1) \lra 9\O \lra \F(1) \lra 0.
\]
Using the Euler sequence the above leads to the resolution
\[
0 \lra 3\O(-1) \oplus 9\O \stackrel{\psi}{\lra} 9\O \oplus 3\O(1) \lra \F \lra 0.
\]
The rank of $\psi_{12}$ is at least $7$, otherwise $\F(1)$ would map surjectively
onto the cokernel of a morphism $3\O(-1) \to 3\O$, in violation of semi-stability.
Canceling $7\O$ and tensoring with $\O(-1)$ we obtain $\f \in W$ such that $\F \isom \Coker(\f)$.

Thus far we have proved that $W \to X$ is a categorical quotient map, so the same is true
for the restricted map $W^\st \to X^\st$. The fibres of the map $W^\st \to X^\st$ are precisely
the $G$-orbits and $X^\st$ is smooth, so we can apply
\cite{popov-vinberg}, theorem 4.2,
to conclude that the map $W^\st \to X^\st$ is a geometric quotient.
\end{proof}

\noi
Let $\W_0 = \Hom(3\O(-2), 3\O)$ and let $W_0 \subset \W_0$ be the set of injective morphisms.
Let
\[
G_0 = (\Aut(3\O(-2)) \times \Aut(3\O))/\C^* = (\GL(3,\C) \times \GL(3,\C))/\C^*
\]
be the natural group acting by conjugation on $\W_0$.
Let $X_0 \subset \M(6,3)$ be the open dense subset
of stable-equivalence classes of sheaves $\F$ as in \ref{2.1}(i).

\begin{prop}
\label{2.3}
There exists a good quotient $W_0/\!/G_0$, which is a proper open subset of $\N(6,3,3)$.
Moreover, $W_0/\!/G_0$ is isomorphic to $X_0$, hence $\M(6,3)$ and $\N(6,3,3)$ are birational.
\end{prop}

\begin{proof}
Let $\W_0^\ss \subset \W_0$ be the subset of morphisms that are semi-stable for the action of $G_0$.
The good quotient $\W_0^\ss/\!/G_0$ constructed using geometric invariant theory is the Kronecker
moduli space $\N(6,3,3)$. According to King's criterion of semi-stability \cite{king},
a morphism $\f \in \W_0$ is semi-stable if and only if it is not equivalent to a morphism having
one of the following forms:
\[
\left[
\ba{ccc}
0 & 0 & 0 \\
\star & \star & \star \\
\star & \star & \star
\ea
\right], \qquad \qquad \left[
\ba{ccc}
0 & 0 & \star \\
0 & 0 & \star \\
\star & \star & \star
\ea
\right], \qquad \qquad \left[
\ba{ccc}
0 & \star & \star \\
0 & \star & \star \\
0 & \star & \star
\ea
\right].
\]
Thus injective morphisms are semi-stable. In fact, it is easy to see that $W_0$
is the preimage in $\W_0^\ss$ under the quotient map of a proper open subset
inside $\W_0^\ss/\!/G_0$. This open subset is the good quotient $W_0/\!/G_0$.
The argument at \ref{2.2} above shows that the canonical map $W_0 \to X_0$
is a categorical quotient map. The isomorphism $X_0 \isom W_0/\!/G_0$
follows from the uniqueness of the categorical quotient.
\end{proof}


\section{The codimension $1$ stratum}

\noi
Let $\W_1 = \Hom(3\O(-2) \oplus \O(-1), \O(-1) \oplus 3\O)$ and let $W_1 \subset \W_1$
be the set of morphisms $\f$ from 2.1(ii). Let $W_1^\st \subset W_1$ be the open subset
of morphisms having stable cokernel. Let
\[
G_1 = (\Aut(3\O(-2) \oplus \O(-1)) \times \Aut(\O(-1) \oplus 3\O))/\C^*
\]
be the natural group acting by conjugation on $\W_1$. Let $W_{10} \subset W_1$
be the set of morphisms $\f$ such that the entries of $\f_{11}$ span $V^*$ and ditto
for the entries of $\f_{22}$. Let $W_{11} \subset W_1$ be the subset given by the condition
that the entries of $\f_{11}$ span a two-dimensional vector subspace of $V^*$.
We denote by $X_1$, $X_1^\st$, $X_{10}$, $X_{11}$
the images of $W_1$, $W_1^\st$, $W_{10}$, $W_{11}$ in $\M(6,3)$.
Clearly $X_1 \setminus X_{10} = X_{11} \cup X_{11}^\D$.
All sheaves giving points in $X_{10}$ are stable.

\begin{prop}
\label{3.1}
There exists a geometric quotient of $W_{10}$ by $G_1$, which is a proper open subset of $\P^{36}$.
Moreover, $W_{10}/G_1$ is isomorphic to $X_{10}$.
In particular, $X_1$ is irreducible and has codimension $1$.

The canonical map $W_1^\st \to X_1^\st$ is a geometric quotient for the action of $G_1$.
Thus $X_{10}$ is an open subset of $X_1$.
\end{prop}

\begin{proof}
Let $W_1' \subset \W_1$ be the subset of morphisms satisfying the conditions defining $W_{10}$
except injectivity. Let $\Sigma \subset W_1'$ be the $G_1$-invariant subset given by the condition
\[
\f_{21} = v \f_{11} + \f_{22} u, \qquad u \in \Hom(3\O(-2), \O(-1)), \qquad v \in \Hom(\O(-1),3\O).
\]
Clearly $W_{10}$ is a proper open $G_1$-invariant subset of $W_1' \setminus \Sigma$.
As at 2.2.2 \cite{mult_five}, it can be shown that there exists a geometric quotient
of $W_1' \setminus \Sigma$ modulo $G_1$, which is a fibre bundle over $\N(3,3,1) \times \N(3,1,3)$
with fibre $\P^{36}$. The base is a point.
Thus $W_{10}/G_1$ exists and is a proper open subset of $(W_1' \setminus \Sigma)/G_1 \isom \P^{36}$.
The argument at \ref{2.2} above shows that the map $W_1 \to X_1$ is a categorical quotient map.
Thus $W_{10} \to X_{10}$ is also a categorical quotient map, so $W_{10}/G_1 \isom X_{10}$.

According to \cite{mumford}, remark (2), p. 5, $X_1$ is normal.
The map $W_1^\st \to X_1^\st$ is a categorical quotient, its fibres are precisely the
$G_1$-orbits and $X_1^\st$ is normal. Thus we can apply
\cite{popov-vinberg}, theorem 4.2,
to conclude that the map $W_1^\st \to X_1^\st$ is a geometric quotient.
Clearly $W_{10}$ is an open $G_1$-invariant subset of $W_1^\st$,
hence $X_{10}$ is an open subset of $X_1^\st$, which is an open subset of $X_1$.
\end{proof}

\begin{prop}
\label{3.2}
The sheaves giving points in $X_{11}$ are either non-split extension sheaves of the form
\[
0 \lra \E \lra \F \lra \C_x \lra 0,
\]
where $\C_x$ is the structure sheaf of a closed point $x \in \P^2$ and $\E$ gives a point in the stratum
$X_2$ of $\M(6,2)$ (cf. 3.3 \cite{mult_six_two}), or they are extension sheaves of the form
\[
0 \lra \O_C \lra \F \lra \G \lra 0,
\]
where $C \subset \P^2$ is a conic curve and $\G$ gives a point in the stratum $X_1$ of $\M(4,2)$
(cf. 4.2.3 \cite{drezet-maican}). Conversely, any such extensions give points in $X_{11}$.

The generic sheaves in $X_{11}$ are of the form $\O_S(3)(-P_1- \cdots - P_7 + P_8)$,
where $S \subset \P^2$ is a smooth sextic curve, $P_i$ are eight distinct points on $S$,
$P_1, \ldots, P_7$ are not contained in a conic curve and no four points among them are colinear.
\end{prop}

\begin{proof}
Let $\F = \Coker(\f)$ give a point in $X_{11}$, where $\f \in W_{11}$ is a morphism represented
by the matrix
\[
\left[
\ba{cccc}
u_1 & u_2 & 0 & 0 \\
\star & \star & q_1 & \ell_1 \\
\star & \star & q_2 & \ell_2 \\
\star & \star & q_3 & \ell_3
\ea
\right]
\]
and let (*) denote the property that the matrix
\[
\left[
\ba{cc}
q_1 & \ell_1 \\
q_2 & \ell_2 \\
q_3 & \ell_3
\ea
\right] \quad \text{be equivalent to a matrix of the form} \quad \left[
\ba{cc}
0 & v_1 \\
0 & v_2 \\
q& v_3
\ea
\right].
\]
Assume that (*) holds. Then $v_1, v_2$ are linearly independent and from the snake lemma
we get an extension
\[
0 \lra \O_C \lra \F \lra \G \lra 0,
\]
where $C \subset \P^2$ is the conic curve given by the equation $q=0$ and $\G$ has resolution
\[
0 \lra 2\O(-2) \oplus \O(-1) \stackrel{\psi}{\lra} \O(-1) \oplus 2\O \lra \G \lra 0,
\]
\[
\psi = \left[
\ba{ccc}
u_1 & u_2 & 0 \\
\star & \star & v_1 \\
\star & \star & v_2
\ea
\right].
\]
By 4.2.3 \cite{drezet-maican}, $\G$ gives a point in the stratum $X_1$ of $\M(4,2)$.
Conversely, from the horseshoe lemma we see that any extension of $\G$ by $\O_C$
is a sheaf in $X_{11}$.
Assume now that (*) is not fulfilled. From the snake lemma we have an extension
\[
0 \lra \E \lra \F \lra \C_x \lra 0,
\]
where $x \in \P^2$ is the point given by the equations $u_1=0$, $u_2=0$
and $\E$ has resolution
\[
0 \lra \O(-3) \oplus \O(-2) \oplus \O(-1) \stackrel{\psi}{\lra} 3\O \lra \E \lra 0,
\]
\[
\psi = \left[
\ba{ccc}
\star & q_1 & \ell_1 \\
\star & q_2 & \ell_2 \\
\star & q_3 & \ell_3
\ea
\right].
\]
From the facts that $\f$ is in $W_1$ and that (*) does not hold we see that $\psi$
satisfies the conditions of 3.3 \cite{mult_six_two}, i.e. that $\E$ belongs to the stratum $X_2$
of $\M(6,2)$.
Conversely, given an extension of $\C_x$ by $\E$,
we combine the above resolution of $\E$ with the resolution
\[
0 \lra \O(-3) \lra 2\O(-2) \lra \O(-1) \lra \C_x \lra 0
\]
to obtain the resolution
\[
0 \lra \O(-3) \lra \O(-3) \oplus 3\O(-2) \oplus \O(-1) \lra \O(-1) \oplus 3\O \lra \F \lra 0.
\]
We may cancel $\O(-3)$ as in the proof of 2.3.3 \cite{mult_five}.
Thus $\F$ is the cokernel of a morphism in $W_{11}$.

The statement about generic sheaves follows from 3.6 \cite{mult_six_two},
where a description of generic sheaves in the stratum $X_2$ of $\M(6,2)$ can be found.
\end{proof}

\begin{remark}
\label{3.3}
If $\F$ is stable-equivalent to $\O_C \oplus \F'$ for a conic curve $C \subset \P^2$
and a sheaf $\F'$ giving a point in the stratum $X_0$ of $\M(4,2)$ (cf. 4.1.1 \cite{drezet-maican}),
then $\F$ gives a point in $X_0$.
If $\F$ is stable-equivalent to $\O_C \oplus \G$, where $\G$ gives a point in the stratum $X_1$
of $\M(4,2)$ (cf. 4.2.3 op.cit.), then, as we saw at \ref{3.2}, $\F$ gives a point in $X_1 \setminus X_{10}$.
If $\F$ is stable-equivalent to $\O_C \oplus \O_Q(1)$, where $Q \subset \P^2$ is a quartic curve,
then $\h^0(\F \tensor \Om^1(1))= 3$ and either $\h^0(\F(-1))=1$ or $\h^1(\F)=1$.
This covers all possible properly semi-stable sheaves $\F$ in $\M(6,3)$.
In particular, we see that the strata $X_0$ and $X_1$ are disjoint.
\end{remark}


\section{The codimension $4$ stratum}

\noi
Recall the vector space $\W$ and the group $G$ from section 2.
For the sake of uniformity of notations denote $\W_2=\W$, $G_2=G$.
Let $W_2 \subset \W_2$ denote the subset of morphisms
$\f$ from \ref{2.1}(iii) and let $X_2$ be the image of $W_2$ in $\M(6,3)$.

\begin{prop}
\label{4.1}
There exists a geometric quotient of $W_2$ by $G_2$, which is a proper open subset of a fibre bundle
over $\N(3,3,2) \times \N(3,2,3)$ with fibre $\P^{21}$. Moreover, $W_2/G_2$ is isomorphic to $X_2$.
In particular, $X_2$ has codimension $4$.
\end{prop}

\begin{proof}
Let $W_2' \subset \W_2$ be the locally closed $G_2$-invariant subset of morphisms $\f$
satisfying the following conditions: $\f_{12}=0$ and $\f_{11}$, $\f_{22}$ are semi-stable
as Kronecker modules, i.e. each of them has linearly independent maximal minors.
Let
\[
U_1 \subset \Hom(3\O(-2), 2\O(-1)), \qquad U_2 \subset \Hom(2\O(-1), 3\O)
\]
denote the subsets of semi-stable morphisms.
Consider the $G_2$-invariant subset $\Sigma \subset W_2'$ given by the condition
\[
\f_{21}= v \f_{11} + \f_{22} u, \qquad u \in \Hom(3\O(-2), 2\O(-1)), \quad v \in \Hom(2\O(-1), 3\O).
\]
Note that $W_2$ is the set of injective morphisms inside $W_2' \setminus \Sigma$.
Clearly $W_2'$ is the trivial vector bundle over $U_1 \times U_2$ with fibre $\Hom(3\O(-2), 3\O)$.
Assume that $\Sigma$ is a sub-bundle. Then the argument at 2.2.2 \cite{mult_five}
shows that the quotient bundle $W_2'/\Sigma$ descends to a vector bundle $F$ over
\[
(U_1/\GL(3,\C) \times \GL(2,\C)) \times (U_2/\GL(2,\C) \times \GL(3,\C)) = \N(3,3,2) \times \N(3,2,3).
\]
Moreover, $\P(F)$ is the geometric quotient of $W_2' \setminus \Sigma$ modulo $G_2$.
Thus $W_2/G_2$ exists as a proper open subset of $\P(F)$.
At \ref{2.2} above we showed that the map $W_2 \to X_2$ is a categorical quotient map.
Thus $W_2/G_2$ is isomorphic to $X_2$.

It remains to show that $\Sigma$ is a sub-bundle of $W_2'$.
Given $(\f_{11},\f_{22}) \in U_1 \times U_2$, let $K(\f_{11}, \f_{22})$ denote the vector space
of pairs $(u,v)$ of morphisms as above, satisfying the relation $v \f_{11} + \f_{22} u =0$.
We must show that the dimension of $K(\f_{11},\f_{22})$ is independent of the choice
of $(\f_{11},\f_{22})$. Assume first that the maximal minors of $\f_{11}$, denoted $\z_1, \z_2, \z_3$,
have no common factor. Let $Z \subset \P^2$ be the scheme given by the ideal $(\z_1, \z_2, \z_3)$.
It is well-known that there is an exact sequence
\[
0 \lra 2\O(-3) \stackrel{\f_{11}^\T}{\lra} 3\O(-2) \stackrel{\z}{\lra} \O \lra \O_Z \lra 0,
\]
\[
\z = \left[
\ba{ccc}
\z_1 & - \z_2 & \phantom{-} \z_3
\ea
\right].
\]
Given $(u,v) \in K(\f_{11},\f_{22})$, we have the relation $\f_{22} u \z^\T = - v \f_{11} \z^\T = 0$,
which, in view of the fact that $\f_{22}$ is injective, yields the relation $u \z^\T=0$.
From the above exact sequence we deduce that $u = \a \f_{11}$ for some
$\a \in \Hom(2\O(-1), 2\O(-1))$.
We have the relation $(v+ \f_{22}\a) \f_{11}=0$, which, in view of the fact that $\f_{11}$ is generically
surjective, yields the relation $v = -\f_{22} \a$. Thus $K(\f_{11},\f_{22})$ is the space of pairs
of the form $(\a \f_{11}, - \f_{22} \a)$, so it has dimension $4$.
By symmetry, the same is true if the maximal minors of $\f_{22}$ have no common factor.

It remains to examine the case when the maximal minors of $\f_{11}$ have a common linear
factor and ditto for the maximal minors of $\f_{22}$. We may write $\f_{11}= \a \Phi$,
$\f_{22}= \Psi \b$, where $\a$, $\b$ are matrices with scalar entries of rank $2$ and
\[
\Phi = \left[
\ba{ccc}
-Y & \phantom{-} X & 0 \\
-Z & \phantom{-} 0 & X \\
\phantom{-} 0 & -Z & Y
\ea
\right], \qquad \qquad \Psi = \left[
\ba{ccc}
-S & -T & \phantom{-} 0 \\
\phantom{-} R & \phantom{-} 0 & - T \\
\phantom{-} 0 & \phantom{-} R & \phantom{-} S
\ea
\right],
\]
where $\{ X, Y, Z \}$ and $\{ R, S, T \}$ are bases of $V^*$. Write
\[
\xi = \left[
\ba{c}
X \\ Y \\ Z
\ea
\right], \qquad \qquad \rho = \left[
\ba{ccc}
R & S & T
\ea
\right].
\]
Let $(u, v)$ belong to $K(\f_{11}, \f_{22})$. Since $\f_{22}$ is injective and
$
\f_{22} u \xi = -v \f_{11} \xi = -v \a \Phi \xi = 0,
$
we get the relation $u \xi =0$. Thus $u= \a' \Phi$ for some matrix $\a' = (a_{ij})$
with scalar entries.
Analogously we have $\rho v = 0$, hence $v = \Psi \b'$ for some matrix $\b' = (b_{kl})$
with scalar entries.
Put $\g = \b' \a + \b \a'$. We have $\Psi \g \Phi = 0$, hence
\[
\g \Phi = \left[
\ba{c}
\phantom{-} T \\
- S \\
\phantom{-} R
\ea
\right] \left[
\ba{ccc}
c_1 & c_2 & c_3
\ea
\right]
\]
for some $c_1, c_2, c_3 \in \C$. Assume that $\g \neq 0$. There are $g, g' \in \GL(3,\C)$
such that
\[
g' \g g = \left[
\ba{ccc}
1 & 0 & 0 \\
\star & \star & \star \\
\star & \star & \star
\ea
\right]. \qquad \text{Thus} \quad g' \g \Phi = g' \left[
\ba{c}
\phantom{-} T \\
- S \\
\phantom{-} R
\ea
\right] \left[
\ba{ccc}
c_1 & c_2 & c_3
\ea
\right] = \left[
\ba{ccc}
1 & 0 & 0 \\
\star & \star & \star \\
\star & \star & \star
\ea
\right] g^{-1} \Phi
\]
showing that
\[
g^{-1} \Phi = \left[
\ba{ccc}
c_1 \ell & c_2 \ell & c_3 \ell \\
\star & \star & \star \\
\star & \star & \star
\ea
\right]
\]
for some non-zero $\ell \in V^*$.
But $g^{-1} \Phi \xi = 0$, hence $c_1 X + c_2 Y + c_3 Z = 0$.
This is absurd. We have proved that $\g = 0$.
There are $g_1, g_2 \in \GL(3,\C)$ such that
\[
\a = \left[
\ba{ccc}
1 & 0 & 0 \\
0 & 1 & 0
\ea
\right] g_1, \qquad \qquad \b = g_2 \left[
\ba{cc}
1 & 0 \\
0 & 1 \\
0 & 0
\ea
\right].
\]
The relation $\b' \a = - \b \a'$ is equivalent to the relation
\[
\left[
\ba{ccc}
a_{11} & a_{12} & a_{13} \\
a_{21} & a_{22} & a_{23} \\
0 & 0 & 0
\ea
\right] g_1^{-1} = - g_2^{-1} \left[
\ba{ccc}
b_{11} & b_{12} & 0 \\
b_{21} & b_{22} & 0 \\
b_{31} & b_{32} & 0
\ea
\right].
\]
The above matrix has the form
\[
\left[
\ba{ccc}
c_{11} & c_{12} & 0 \\
c_{21} & c_{22} & 0 \\
0 & 0 & 0
\ea
\right]. \qquad \text{Thus} \quad \a' = \left[
\ba{ccc}
c_{11} & c_{12} & 0 \\
c_{21} & c_{22} & 0
\ea
\right] g_1, \quad \b' = - g_2 \left[
\ba{cc}
c_{11} & c_{12} \\
c_{21} & c_{22} \\
0 & 0
\ea
\right].
\]
We conclude that $K(\f_{11}, \f_{22})$ is parametrised by the quadruple $(c_{ij})$,
so it is a vector space of dimension $4$.
\end{proof}

\begin{prop}
\label{4.2}
The generic sheaves giving points in $X_2$ have the form
\[
\O_C(2)(-P_1-P_2-P_3+P_4+P_5+P_6),
\]
where $C \subset \P^2$ is a smooth sextic curve,
$P_i$ are six distinct points on $C$, $P_1, P_2, P_3$ are non-colinear,
$P_4, P_5, P_6$ are also non-colinear.
In particular, $X_2$ is contained in the closure of $X_{11}$ and also in the closure of $X_{11}^\D$.
\end{prop}

\begin{proof}
Given a morphism $\f \in W_2$ we denote by $\z_1, \z_2, \z_3$ the maximal minors of $\f_{11}$
and by $\upsilon_1, \upsilon_2, \upsilon_3$ the maximal minors of $\f_{22}$.
Let $W_{20} \subset W_2$ be the open $G_2$-invariant subset given by the following properties:
$\z_1, \z_2. \z_3$ generate the ideal of a reduced zero-dimensional scheme $Z \subset \P^2$,
$\upsilon_1, \upsilon_2, \upsilon_3$ generate the ideal of a reduced zero-dimensional scheme
$Y \subset \P^2$, $Z$ and $Y$ have no common points, the equation $\det(\f)=0$
determines a smooth sextic curve $C \subset \P^2$.
Let $X_{20} \subset X_2$ be the image of $W_{20}$ in $\M(6,3)$.
If $\F$ gives a point in $X_{20}$, then, from the snake lemma, we get an extension
\[
0 \lra \O_C(2)(-Y) \lra \F \lra \O_Z \lra 0.
\]
Now $Y$ is the union of three non-colinear points $P_1, P_2, P_3$ and
$Z$ is the union of three non-colinear points $P_4, P_5, P_6$ distinct from
$P_1, P_2, P_3$. Thus
\[
\F \isom \O_C(2)(-P_1-P_2-P_3+P_4+P_5+P_6).
\]
Conversely, consider a sheaf $\F$ as above. Clearly $\F$ is stable and gives a point
in $\M(6,3)$. Our aim is to show that $\F$ gives a point in $X_{20}$.

We claim that $\H^1(\F)=0$. Denote $\F'= \O_C(2)(-P_1-P_2-P_3)$.
According to \cite{modules-alternatives}, propositions 4.5 and 4.6,
we have an exact sequence
\[
0 \lra \O(-4) \oplus 2\O(-1) \lra 3\O \lra \F' \lra 0
\]
(compare with 2.3.4(i) \cite{mult_five}). Thus $\h^0(\F')=3$.
Consider the exact sequence
\[
0 \lra \F' \lra \F \lra \O_Z \lra 0
\]
and let $\d \colon \H^0(\O_Z) \to \H^1(\F')$ denote the connecting homomorphism.
To prove that $\h^0(\F)=3$ it is enough to show that $\d$ is injective or, equivalently,
that its dual $\d^*$ is surjective.
By Serre duality $\d^*$ is the restriction morphism
\[
\xymatrix
{
\H^0(\O_C(-2)(P_1+P_2+P_3) \tensor \omega_C) \ar[r] \egal[d]
&
\H^0((\O_C(-2)(P_1+P_2+P_3) \tensor \omega_C)_{|Z}) \egal[d]
\\
\H^0(\O_C(1)(P_1+P_2+P_3)) \egal[d]
&
\H^0(\O_C(1)(P_1+P_2+P_3)_{|Z}) \egal[d]
\\
\H^0(\O_C(1)) & \H^0(\O_C(1)_{|Z})
}.
\]
The identity $\H^0(\O_C(1)(P_1+P_2+P_3)) \isom \H^0(\O_C(1)) \isom V^*$
follows from the exact sequence
\[
0 \lra 3\O(-3) \lra 2\O(-2) \oplus \O(1) \lra \O_C(1)(P_1+P_2+P_3) \lra 0
\]
obtained by dualising the above resolution of $\F'$.
Let $\e_i \colon \H^0(\O_Z) \to \C$ be the linear form of evaluation at $P_i$, $i=4,5,6$.
We see from the above diagram that, given a one-form $u$, $\d^*(u)$ is a multiple of $\e_i$
precisely if the line given by the equation $u=0$ does not pass through $P_i$ but passes
through the other two points.
By hypothesis $P_4, P_5, P_6$ are non-colinear, hence such a line exists.
We conclude that each $\e_i$ is in the image of $\d^*$, so this map is surjective.

Thus far we have shown that $\H^1(\F)=0$. By duality $\H^0(\F(-1))$ also vanishes.
By proposition \ref{2.1}, $\F$ gives a point in $X_0 \cup X_1 \cup X_2$.
Notice that the subsheaf of $\F$ generated by its global sections is $\F'$.
If $\F$ gave a point in $X_0$, then $\F$ would be generated by its global sections,
which is not the case. If $\F$ gave a point in $X_1$, then $\F/\F'$ would be the zero-sheaf
or the structure sheaf of a closed point. This, again, is not the case.
Thus $\F= \Coker(\f)$ for some $\f \in W_2$.
Assume that $\z_1, \z_2, \z_3$ have a common linear factor.
Then $\Coker(\f_{11}) \isom \O_L$ for a line $L \subset \P^2$
and from the snake lemma we obtain the exact sequences
\[
0 \lra \F'' \lra \F \lra \O_L \lra 0,
\]
\[
0 \lra \O(-3) \oplus 2\O(-1) \lra 3\O \lra \F'' \lra 0.
\]
Notice that $\F''$ is generated by its global sections and $\H^0(\F'') = \H^0(\F)$.
It follows that $\F''=\F'$, hence $\O_L \isom \O_Z$, which is absurd.
This proves that $\z_1, \z_2, \z_3$ have no common factor, i.e. they generate
the ideal of a zero-dimensional scheme.
In point of fact, the above argument shows that $\z_1, \z_2, \z_3$
generate the ideal of $Z$.
By duality $\upsilon_1, \upsilon_2, \upsilon_3$ generate the ideal of $Y$.
The curve $C$ has equation $\det(\f)=0$ and is smooth by hypothesis.
We conclude that $\f$ belongs to $W_{20}$, i.e. that $\F$ gives a point in $X_{20}$.

To prove the inclusion $X_2 \subset \overline{X}_{11}$ fix a generic sheaf in $X_2$
as in the proposition. We may assume that the line $P_4 P_5$ meets $C$ at six
distinct points $P_4, P_5, Q_4, Q_5, Q_6, Q_7$. Denote $Q_1 = P_1$, $Q_2 = P_2$,
$Q_3 = P_3$. Thus
\[
\O_C(2)(-P_1 -P_2 -P_3 + P_4 +P_5 + P_6) \isom
\O_C(3)(-Q_1 - \cdots - Q_7 + P_6).
\]
Choose distinct points $R_i$, $1 \le i \le 7$, on $C$, that are also distinct from $P_6$
and satisfy the conditions of \ref{3.2},
i.e. they do not lie on a conic curve and no four of them are colinear.
Then, according to loc.cit.,
\[
\O_C(3)(-R_1 - \cdots - R_7 + P_6)
\]
gives a point in $X_{11}$. Making $R_i$ converge to $Q_i$ for $1 \le i \le 7$
we obtain a sequence of points converging to the fixed generic point of $X_2$.
Thus $X_2 \subset \overline{X}_{11}$.
The inclusion $X_2 \subset \overline{X}_{11}^\D$ follows from the fact that $X_2$
is self-dual.
\end{proof}

\begin{prop}
\label{4.3}
\emph{(i)} Let $\F$ be a sheaf giving a point in $\M(6,3)$ and satisfying the conditions
$\h^0(\F(-1))=0$, $\h^1(\F)=1$. Then $\h^0(\F \tensor \Om^1(1))=3$.
These sheaves are precisely the sheaves having resolution of the form
\[
0 \lra \O(-3) \oplus 3\O(-1) \stackrel{\f}{\lra} 4\O \lra \F \lra 0,
\]
where $\f_{12}$ is semi-stable as a Kronecker $V$-module.

\medskip

\noi
\emph{(ii)} Let $\F$ be a sheaf giving a point in $\M(6,3)$ and satisfying the conditions
$\h^0(\F(-1))=1$, $\h^1(\F)=0$. Then $\h^0(\F \tensor \Om^1(1))=3$.
These sheaves are precisely the sheaves having resolution of the form
\[
0 \lra 4\O(-2) \stackrel{\f}{\lra} 3\O(-1) \oplus \O(1) \lra \F \lra 0,
\]
where $\f_{11}$ is semi-stable as a Kronecker $V$-module.
\end{prop}

\begin{proof}
Part (i) is a particular case of 5.3 \cite{maican}.
Part (ii) is equivalent to (i) by duality.
\end{proof}

\noi
Let $\W_3 = \Hom(\O(-3) \oplus 3\O(-1), 4\O)$ and let $W_3 \subset \W_3$ be the open subset of
morphisms as in \ref{4.3}(i). The group
\[
G_3 = (\Aut(\O(-3) \oplus 3\O(-1)) \times \Aut(4\O))/\C^*
\]
acts on $\W_3$ by conjugation and leaves $W_3$ invariant.
Let $\W_3^\D= \Hom(4\O(-2), 3\O(-1) \oplus \O(1))$ and let $W_3^\D \subset \W_3^\D$
be the subset of morphisms $\f$ as in \ref{4.3}(ii). The natural group acting on $\W_3^\D$
is denoted $G_3^\D$. Let $X_3, X_3^\D$ be the images of $W_3, W_3^\D$ in $\M(6,3)$.
These notations are justified because the morphisms of $W_3^\D$ are the transposes of the morphisms
in $W_3$, hence $X_3^\D$ is the image of $X_3$ under the duality automorphism of $\M(6,3)$.
Consider the vector spaces
\[
\U = \Hom(3\O(-1), 4\O), \qquad \qquad \U^\D = \Hom(4\O(-2), 3\O(-1)).
\]
Let $U \subset \U$, $U^\D \subset \U^\D$ denote the subsets of morphisms that are semi-stable
as Kronecker $V$-modules. According to 3.3 \cite{mult_five}, the kernel of a morphism in $U^\D$
is isomorphic to $\O(-5)$, $\O(-4)$ or $\O(-3)$.
The subset $U_0^\D \subset U^\D$ of morphisms $\psi$ for which $\Ker(\psi) \isom \O(-5)$ is open.
We denote by $U_1^\D, U_2^\D$ the subsets of morphisms $\psi$ for which $\Ker(\psi) \isom \O(-4)$,
respectively $\O(-3)$. The counterparts in $U$ of these subsets are denoted $U_0, U_1, U_2$.
Let $W_{3i} \subset W_3$, $i=0,1,2$, be the subset of those morphisms $\f$ for which $\f_{12} \in U_i$
and let $X_{3i}$ be its image in $\M(6,3)$. Analogously we define $W_{3i}^\D$ and $X_{3i}^\D$.

\begin{prop}
\label{4.4}
The sheaves giving points in $X_{30}$ are precisely the sheaves of the form $\J_Z(3)$,
where $Z \subset \P^2$ is a zero-dimensional scheme of length $6$ not contained in a conic curve,
contained in a sextic curve $C$, and $\J_Z \subset \O_C$ is its ideal sheaf.

The generic sheaves in $X_3$ have the form $\O_C(3)(-P_1 - \cdots - P_6)$,
where $C \subset \P^2$ is a smooth sextic curve and $P_i$ are six distinct points on $C$
that are not contained in a conic curve. The generic sheaves in $X_3^\D$ have the form
$\O_C(1)(P_1+ \cdots +P_6)$.
In particular, $X_3$ is contained in the closure of $X_{11}$ and $X_3^\D$ is contained in the
closure of $X_{11}^\D$.
Thus $X_3$ and $X_3^\D$ are contained in the closure of $X_1$.
\end{prop}

\begin{proof}
By \cite{modules-alternatives}, propositions 4.5 and 4.6, the cokernels of morphisms in $U_0$
are precisely the twisted ideal sheaves $\I_Z(3) \subset \O(3)$ of zero-dimensional schemes
$Z \subset \P^2$ of length $6$ that are not contained in conic curves.
The first statement now follows as at 2.3.4(i) \cite{mult_five}.

To prove the inclusion $X_3 \subset \overline{X}_{11}$ we use the form of generic sheaves in $X_{11}$
found at \ref{3.2}. Fix a generic point in $X_3$ represented by $\O_C(3)(-P_1- \cdots -P_6)$.
Notice that no four points among $P_i$ are colinear. We can thus choose points $P_7, P_8 \in C$
such that $P_1, \ldots, P_8$ satisfy the conditions of loc.cit.
Thus $\O_C(3)(-P_1 - \cdots - P_7 + P_8)$ gives a point in $X_{11}$.
Making $P_8$ converge to $P_7$ we obtain a sequence of points in $X_{11}$
converging to the fixed generic point of $X_3$.
\end{proof}

\begin{prop}
\label{4.5}
\emph{(i)} The sheaves of the form $\Coker(\f)$, $\f \in W_{32}^\D$, are precisely the extension
sheaves of the form
\[
0 \lra \O_Q(1) \lra \F \lra \O_C \lra 0,
\]
satisfying the condition $\H^1(\F)=0$. Here $Q$ and $C$ are arbitrary quartic, respectively conic curves
in $\P^2$.

\medskip

\noi
\emph{(ii)} The sheaves of the form $\Coker(\f)$, $\f \in W_{32}$, are precisely the extension sheaves
of the form
\[
0 \lra \O_C \lra \F \lra \O_Q(1) \lra 0,
\]
satisfying the condition $\H^0(\F(-1))=0$.

\medskip

\noi
Thus $X_{32}$ coincides with $X_{32}^\D$ and consists of all stable-equivalence classes
$\O_Q(1) \oplus \O_C$. Moreover, $X_3 \cap X_3^\D = X_{32}$.
All sheaves giving points in $X_{30}$, $X_{31}$, $X_{30}^\D$, $X_{31}^\D$ are stable.
\end{prop}

\begin{proof}
The proof of (i) is nearly identical to the proof of 3.3.2 \cite{mult_five} and is based on the fact that
the cokernel of any morphism $\psi \in U_2^\D$ is isomorphic to $\O_C$ for some conic curve
$C \subset \P^2$ and, conversely, any $\O_C$ is isomorphic to $\Coker(\psi)$ for some
$\psi \in U_2^\D$.
Part (ii) is equivalent to (i) by duality.

Let $Q \subset \P^2$ be a quartic curve given by the equation $h=0$ and let $C \subset \P^2$
be a conic curve with equation $g=0$.
We can choose $\ell_1, \ell_2, \ell_3 \in V^*$ and $f_1, f_2, f_3 \in \SS^3 V^*$ such that
\[
X \ell_3 - Y \ell_2 + Z \ell_1 = g \qquad \text{and} \qquad X f_1 + Y f_2 + Z f_3 = h.
\]
The cokernel of any morphism $\f \in W_{32}^\D$ represented by a matrix of the form
\[
\left[
\ba{cccc}
\ell_1 & -Y & \phantom{-} X & 0 \\
\ell_2 & -Z & \phantom{-} 0 & X \\
\ell_3 & \phantom{-} 0 & -Z & Y \\
f & \phantom{-} f_1 & \phantom{-} f_2 & f_3
\ea
\right]
\]
is stable-equivalent to $\O_Q(1) \oplus \O_C$. This proves that $X_{32}$ coincides with $X_{32}^\D$
and consists of all stable equivalence classes $\O_Q(1) \oplus \O_C$.
According to remark \ref{3.3}, any point of $X_3$ or of $X_3^\D$ represented by a properly semi-stable
sheaf must be in $X_{32}$. It is now easy to see that $X_{30}$, $X_{31}$, $X_{30}^\D$, $X_{31}^\D$
contain only stable sheaves, so these sets are disjoint, so $X_3 \cap X_3^\D = X_{32}$.
For example, assume that $\F = \Coker(\f)$, $\f \in W_{30} \cup W_{31}$, is stable-equivalent to
$\O_Q(1) \oplus \O_C$. Since $\H^0(\F(-1))=0$, $\O_Q(1)$ cannot be a subsheaf of $\F$.
By part (ii) of the proposition $\F = \Coker(\f')$ for some $\f' \in W_{32}$.
It is easy to see that $\f$ and $\f'$ must be in the same $G_3$-orbit, which is absurd.
\end{proof}

\begin{prop}
\label{4.6}
There exists a geometric quotient of $W_3$ modulo $G_3$, which is an open subset
of a fibre bundle over $\N(3,3,4)$ with fibre $\P^{21}$.
The image of $W_{30} \cup W_{31}$ in $W_3/G_3$ is the geometric quotient
$(W_{30} \cup W_{31})/G_3$ and is isomorphic to $X_3^\st=X_{30} \cup X_{31}$.
By duality there exists a geometric quotient $W_3^\D/G_3$ and a certain open subset
of this quotient, namely $(W_{30}^\D \cup W_{31}^\D)/G_3$, is isomorphic
to $(X_3^\D)^\st=X_{30}^\D \cup X_{31}^\D$.
In particular, $X_3$ and $X_3^\D$ have codimension $4$.
\end{prop}

\begin{proof}
Let $\W^\ss_3(\L) \subset \W_3$ denote the set of morphisms that are semi-stable
with respect to a polarisation $\L= (\l_1, \l_2, \m_1)$ satisfying the relation
$0 < \l_1 < 1/4$ (notations as as \cite{drezet-trautmann}).
Concretely, $\W_3^\ss(\L)$ consists of those morphisms $\f$ for which $\f_{11}$
is semi-stable as a Kronecker $V$-module and $\f_{11} \neq \f_{12} u$
for any $u \in \Hom(\O(-3), 3\O(-1))$.
According to 5.3 \cite{maican}, $W_3$ is the set of injective morphisms inside $\W_3^\ss(\L)$.
According to 9.3 \cite{drezet-trautmann}, there exists a geometric quotient
$\W_3^\ss(\L)/G_3$
and it is a fibre bundle as in the proposition (compare with 3.2.1 \cite{mult_five}).
Thus $W_3/G_3$ exists as a proper open subset of $\W_3^\ss(\L)/G_3$.

We saw at \ref{4.5} above that $W_{30} \cup W_{31}$ is the subset of morphisms
$\f \in W_3$ for which $\Coker(\f)$ is stable. In flat families stability is an open condition,
hence $W_{30} \cup W_{31}$ is an open $G_3$-invariant subset of $W_3$,
hence $(W_{30} \cup W_{31})/G_3$ exists as a proper open subset of $W_3/G_3$.
Notice that $X_{30} \cup X_{31}$ is an open subset of $X_3$, so it inherits the canonical
induced reduced structure.
The canonical morphism
\[
(W_{30} \cup W_{31})/G_3 \lra X_{30} \cup X_{31}
\]
is easily seen to be bijective.
We will show that the inverse of this map is also a morphism by constructing
resolution \ref{4.3}(i) starting from the Beilinson spectral sequence of a sheaf
$\F$ in $X_{30} \cup X_{31}$.
Diagram (2.2.3) \cite{drezet-maican} takes the form
\[
\xymatrix
{
3\O(-2) \ar[r]^-{\f_1} & 3\O(-1) \ar[r]^-{\f_2} & \O \\
0 & 3\O(-1) \ar[r]^-{\f_4} & 4\O
}.
\]
As at 2.2.4 \cite{mult_five}, we have $\Ker(\f_1) \isom \O(-3)$ and $\Ker(\f_2) = \Im(\f_1)$.
The exact sequence (2.2.5) \cite{drezet-maican} now takes the form
\[
0 \lra \O(-3) \lra \Coker(\f_4) \lra \F \lra 0.
\]
Clearly the morphism $\O(-3) \to \Coker(\f_4)$ lifts to a morphism $\O(-3) \to 4\O$.
We obtain $\f \in \W_3$ such that $\Coker(\f) \isom \F$.
Since $\F$ was a priori chosen in $X_{30} \cup X_{31}$, we see that $\f$
belongs to $W_{30} \cup W_{31}$.
\end{proof}

\begin{prop}
\label{4.7}
The sheaves $\F$ giving points in $X_{31}^\D$ are precisely the extension sheaves having
one of the forms
\[
0 \lra \E \lra \F \lra \O_L \lra 0,
\]
\[
0 \lra \G \lra \F \lra \O_L(1) \lra 0,
\]
\[
0 \lra \O_C(1) \lra \F \lra \O_L(2) \lra 0
\]
and satisfying the cohomological conditions $\h^0(\F(-1))=1$, $\h^1(\F)=0$.
Here $L \subset \P^2$ is an arbitrary line, $C \subset \P^2$ is an arbitrary quintic curve,
$\E$ is an arbitrary sheaf giving a point in the stratum $X_3$ of $\M(5,2)$
(cf. 2.3.5 \cite{mult_five}) and $\G$ is an arbitrary sheaf giving a point in the stratum
$X_3$ of $\M(5,1)$ (cf. 3.1.5 op.cit.)
\end{prop}

\begin{proof}
Consider a sheaf $\F = \Coker(\f)$, $\f \in W_{31}^\D$.
Denote $\CC= \Coker(\f_{11})$ and let $\TT$ be the zero-dimensional torsion of $\CC$.
The Hilbert polynomial of $\CC$ is $\PP(t)=t+3$, hence $\CC/\TT$ is isomorphic to $\O_L(d)$
for some line $L \subset \P^2$ and integer $d \le 2$.
From the snake lemma we have an exact sequence
\[
0 \lra \O_C(1) \lra \F \lra \CC \lra 0,
\]
where $C \subset \P^2$ is a quintic curve. Since $\O_L(d)$ is a quotient sheaf of the semi-stable
sheaf $\F$, we see that $d= 0, 1$ or $2$, that is $\length(\TT)=2, 1$ or $0$.

Assume that $\length(\TT)=2$. Let $\E$ be the preimage of $\TT$ in $\F$.
According to 2.3.5 \cite{mult_five}, $\E$ gives a point in the stratum $X_3$ of $\M(5,2)$.
Assume that $\length(\TT)=1$. Let $\G$ be the preimage of $\TT$ in $\F$.
According to 3.1.5 op.cit., $\G$ gives a point in the stratum $X_3$ of $\M(5,1)$.

Conversely, assume that $\F$ is an extension as in the proposition and satisfies the conditions
$\h^0(\F(-1))=1$, $\h^1(\F)=0$. Firstly, we will show that $\F$ is semi-stable.
Let $\F' \subset \F$ be a non-zero subsheaf of multiplicity at most $5$.
There are extensions of the form
\[
0 \lra \O_C(1) \lra \E \lra \O_Z \lra 0,
\]
\[
0 \lra \O_C(1) \lra \G \lra \C_x \lra 0,
\]
where $C \subset \P^2$ is a a quintic curve, $\O_Z$ is the structure sheaf of a zero-dimensional
scheme $Z \subset \P^2$ of length $2$ and $\C_x$ is the structure sheaf of a closed point $x \in \P^2$.
All three possible extensions in the proposition lead to an extension
\[
0 \lra \O_C(1) \lra \F \lra \CC \lra 0,
\]
where the zero-dimensional torsion $\TT$ of $\CC$ has length at most $2$.
If the image of $\F'$ in $\CC$ is a subsheaf of $\TT$, then $\F'$ is a subsheaf of $\E$,
or of $\G$, or of $\O_C(1)$. These three sheaves are stable, hence $\pp(\F')$ is at most
$2/5$, respectively $1/5$, respectively $0$.
Assume now that the image of $\F'$ in $\CC$ is not a subsheaf of $\TT$.
Precisely as at 4.4 \cite{mult_six_two}, we have the equation
\[
\PP_{\F'}(t)= (5-d)t + \frac{d^2-5d}{2} + t + 3 -a
\]
for some integers $a$ and $d$ satisfying the inequalities $a \ge 0$, $1 \le d \le 4$.
Thus
\[
\pp(\F') \le \frac{d^2-5d+6}{2(6-d)} \le \frac{1}{2} = \pp(\F).
\]
Secondly, applying \ref{4.3}(ii), we obtain $\f \in W_3^\D$ such that $\F \isom \Coker(\f)$.
The subsheaf generated by the global sections of $\E(-1)$ is isomorphic to $\O_C$
and the same is true of $\G$. Thus $\H^0(\F(-1))$ generates $\O_C$ in all three possible
cases. We deduce that $\f$ is in $W_{31}^\D$, otherwise $\H^0(\F(-1))$
would generate the structure sheaf of a quartic or of a sextic curve.
\end{proof}


\section{The codimension $5$ stratum}

\begin{prop}
\label{5.1}
Let $\F$ be a sheaf giving a point in $\M(6,3)$ and satisfying the cohomological conditions
$\h^0(\F(-1))=1$, $\h^1(\F)=1$. Then $\h^0(\F \tensor \Om^1(1))=3$ or $4$.
The sheaves in the first case are precisely the sheaves having resolution of the form
\[
0 \lra \O(-3) \oplus \O(-2) \stackrel{\f}{\lra} \O \oplus \O(1) \lra \F \lra 0,
\]
where $\f_{12} \neq 0$.
\end{prop}

\begin{proof}
Assume that $\F$ gives a point in $\M(6,3)$ and satisfies the above cohomological conditions.
Put $m= \h^0(\F \tensor \Om^1(1))$. The Beilinson free monad (2.2.1) \cite{drezet-maican}
for $\F$ has the form
\[
0 \lra \O(-2) \lra 4\O(-2) \oplus m\O(-1) \lra m\O(-1) \oplus 4\O \lra \O \lra 0
\]
and yields the resolution
\[
0 \lra \O(-2) \lra 4\O(-2) \oplus m\O(-1) \lra \Om^1 \oplus (m-3)\O(-1) \oplus 4\O \lra \F \lra 0.
\]
Note that $m \ge 3$. Using the Euler sequence and arguing as at 2.1.4 \cite{mult_five}
we arrive at a resolution
\[
0 \lra \O(-2) \stackrel{\psi}{\lra} \O(-3) \oplus \O(-2) \oplus m\O(-1) \stackrel{\f}{\lra} (m-3)\O(-1)
\oplus 4\O \lra \F \lra 0
\]
in which $\psi_{11}=0$, $\psi_{21}=0$, $\f_{13}=0$ and the entries of $\psi_{31}$ span $V^*$.
From the fact that $\F$ maps surjectively onto $\Coker(\f_{11}, \f_{12})$,
we deduce the inequality $m \le 4$.
In the sequel we will assume that $m=3$.
Thus $\Coker(\psi_{31}) \isom \Om^1(1)$ and we have a resolution
\[
0 \lra \O(-3) \oplus \O(-2) \oplus \Om^1(1) \stackrel{\f}{\lra} 4\O \lra \F \lra 0.
\]
Arguing as at loc.cit., we see that $\Coker(\f_{13}) \isom \O \oplus \O(1)$,
which leads to a resolution as in the proposition.

Conversely, assume that $\F$ has a resolution as in the proposition.
The relations $\h^0(\F(-1))=1$, $h^1(\F)=1$ are obvious, while the relation
$\h^0(\F \tensor \Om^1(1))= 3$ follows from Bott's formulas.
We need to show that $\F$ is semi-stable.
Let $\F' \subset \F$ be a subsheaf of multiplicity at most $5$.
We will distinguish three situations according to the degree of the greatest common divisor
of $\f_{12}$ and $\f_{22}$.

Assume first that $\f_{12}$ and $\f_{22}$ have no common factor.
Let $Z \subset \P^2$ be the zero-dimensional scheme of length $6$ given by the ideal
$(\f_{12}, \f_{22})$. Notice that $\F = \J_Z(3)$, where $\J_Z \subset \O_C$
is the ideal of $Z$ as a subscheme of the sextic curve $C \subset \P^2$
given by the equation $\det(\f)=0$. According to \cite{maican}, lemma 6.7,
there is a sheaf $\A \subset \O_C(3)$ containing $\F'$ such that $\A/\F'$
is supported on finitely many points and $\O_C(3)/\A$ is isomorphic to $\O_S(3)$
for a curve $S \subset C$ of degree $d \le 5$.
The slope of $\F'$ can be estimated as at 2.1.4 \cite{mult_five}:
\begin{align*}
\PP_{\F'}(t) & = \PP_{\A}(t) - \h^0(\A/\F') \\
& = \PP_{\O_C}(t+3) - \PP_{\O_S}(t+3) - \h^0(\A/\F') \\
& = (6-d)t + \frac{(d-6)(d-3)}{2} - \h^0(\A/\F'), \\
\pp(\F') & = \frac{3-d}{2} - \frac{\h^0(\A/\F')}{6-d},
\end{align*}
hence $\pp(\F') \le 1/2$ except, possibly, when $d=1$ and $\h^0(\A/\F') \le 2$.
Thus we need to examine the case when $S$ is the line given by the equation
$\ell = 0$, where $\ell \in V^*$.
If $\A = \F'$, then $\J_S \subset \J_Z$, hence $Z$ is a subscheme of $S$.
From B\'ezout's theorem we see that $\ell$ divides both $\f_{12}$ and $\f_{22}$,
contrary to our hypothesis.
If $\h^0(\A/\F')=1$ or $2$, then $S$ contains a subscheme of $Z$ of length $4$ or $5$
and we get a contradiction as above.

Secondly, assume that $\gcd(\f_{12},\f_{22})=\ell$ for some $\ell \in V^*$.
Let $L \subset \P^2$ be the line given by the equation $\ell=0$.
We have an extension
\[
0 \lra \O_L(-1) \lra \F \lra \E \lra 0,
\]
\[
\E = \Coker(\psi), \quad \psi \in \Hom(\O(-3) \oplus \O(-1), \O \oplus \O(1)), \quad
\f_{12} = \ell \psi_{12}, \quad \f_{22} = \ell \psi_{22}.
\]
According to loc.cit., $\E$ gives a point in $\M(5,3)$.
Let $\E'$ be the image of $\F'$ in $\E$. If $\E' \neq \E$, then $\pp(\E') \le 1/2$,
forcing $\pp(\F') \le 1/2$. Assume that $\E'=\E$. The sheaf $\F' \cap \O_L(-1)$
is zero because $\F'$ was assumed to have multiplicity at most $5$.
Thus the above extension splits, hence, by loc.cit., we have the formula
\[
\h^0(\F \tensor \Om^1(1)) = \h^0(\O_L \tensor \Om^1) + \h^0(\E \tensor \Om^1(1)) = 4,
\]
which is absurd.

Finally, if $\f_{12}$ divides $\f_{22}$, then $\F$ is stable equivalent to $\O_Q(1) \oplus \O_C$
for a quartic curve $Q$ and a conic curve $C$ in $\P^2$.
\end{proof}

\noi
Denote $\W_4 = \Hom(\O(-3) \oplus \O(-2), \O \oplus \O(1))$ and let $W_4 \subset \W_4$
be the subset of injective morphisms $\f$ for which $\f_{12}$ is non-zero and divides
neither $\f_{11}$ nor $\f_{22}$.
Let $W_{41} \subset \W_4$ be the subset of injective morphism $\f$ for which
$\f_{12}$ is non-zero and divides either $\f_{11}$ or $\f_{22}$.
The algebraic group
\[
G_4 = (\Aut(\O(-3) \oplus \O(-2)) \times \Aut(\O \oplus \O(1)))/\C^*
\]
acts by conjugation on $\W_4$ and leaves $W_4$ and $W_{41}$ invariant.
Let $X_4$ and $X_{41}$ be the images of these sets under the map $\f \mapsto [\Coker(\f)]$.

\begin{prop}
\label{5.2}
The semi-stable sheaves on $\P^2$ that are stable-equivalent to $\O_C \oplus \O_Q(1)$
for some conic curve $C$ and quartic curve $Q$ are precisely the cokernels of morphisms
in $W_{32}$, or in $W_{32}^\D$, or in $W_{41}$.
In particular, $X_{41} = X_{32}$ and all sheaves giving points in $X_4$ are stable.
\end{prop}

\begin{proof}
Assume that we have an extension
\[
0 \lra \O_Q(1) \lra \F \lra \O_C \lra 0.
\]
Then $\h^1(\F) =0$ or $1$. In the first case, as seen at \ref{4.5}(i), $\F$ is the cokernel
of some morphism in $W_{32}^\D$.
In the second case we can apply the horseshoe lemma to find $\f \in W_{41}$
such that $\F \isom \Coker(\f)$.
The dual case in which $\F$ is an extension of $\O_Q(1)$ by $\O_C$ is analogous.
\end{proof}

\begin{prop}
\label{5.3}
There exists a geometric quotient of $W_4$ by $G_4$, which is a proper open subset
of a tower of bundles with fibre $\P^{21}$ and base a fibre bundle over $\P^5$ with fibre $\P^6$.
Moreover, $W_4/G_4$ is isomorphic to $X_4$.
In particular, $X_4$ has codimension $5$.
\end{prop}

\begin{proof}
The linear algebraic group $G= \Aut(\O \oplus \O(1))$ acts on the vector space
$\U = \Hom(\O(-2), \O \oplus \O(1))$ by left-multiplication.
Consider the open $G$-invariant subset $U \subset \U$ of morphisms $\psi$
for which $\psi_{11}$ is non-zero and does not divide $\psi_{21}$.
Consider the fibre bundle with base $\P(\SS^2 V^*)$ and fibre $\P(\SS^3 V^*/V^* q)$
at any point of the base represented by $q \in \SS^2 V^*$.
Clearly this fibre bundle is the geometric quotient of $U$ modulo $G$.

Consider the open $G_4$-invariant subset $W_4' \subset \W_4$ of morphisms $\f$
whose restriction to $\O(-2)$ lies in $U$.
Clearly $W_4'$ is the trivial vector bundle over $U$ with fibre $\Hom(\O(-3), \O \oplus \O(1))$.
Consider the sub-bundle $\Sigma \subset W_4'$ given by the conditions
\[
\f_{11} = \f_{12} u, \qquad \f_{21} = \f_{22} u, \qquad u \in \Hom(\O(-3), \O(-2)).
\]
As at 2.2.5 \cite{mult_five}, the quotient bundle $W_4'/\Sigma$ is $G$-linearised, 
hence it descends to a vector bundle $E$ over $U/G$.
Its projectivisation $\P(E)$ is the geometric quotient of $W_4' \setminus \Sigma$
modulo $G_4$.
Notice that $W_4$ is a proper open $G_4$-invariant subset of $W_4' \setminus \Sigma$.
Thus $W_4/G_4$ exists as a proper open subset of $\P(E)$.

The argument at proposition \ref{7.1} below shows that the canonical map $W_4 \to X_4$
is a geometric quotient map.
\end{proof}

\begin{prop}
\label{5.4}
The generic sheaves in $X_4$ have the form $\O_C(3)(-P_1 - \cdots -P_6)$,
where $C \subset \P^2$ is a smooth sextic curve and $P_i$ are six distinct points on $C$
which lie on a conic curve and no four of which lie on a line.
In particular, $X_4$ is contained in the closure of $X_3$ and also in the closure of $X_3^\D$.
\end{prop}

\begin{proof}
Let $W_{40} \subset W_4$ be the subset of morphisms $\f$ satisfying the following conditions:
the curve given by the equation $\det(\f) = 0$ is smooth, $\f_{12}$ and $\f_{22}$ have no common factor
and the curves they determine meet at six distinct points $P_1, \ldots, P_6$.
Notice that no four among these points are colinear.
As already mentioned in the proof of \ref{5.1}, $\Coker(\f)$ is isomorphic to
$\O_C(3)(-P_1 - \cdots - P_6)$ if $\f \in W_{40}$.

Conversely, assume that we are given a sheaf as in the proposition.
Let $A$ be the conic curve passing through $P_1, \ldots, P_6$.
If $A$ is irreducible, then it is easy to find a cubic curve $B$ passing through $P_1, \ldots, P_6$
that does not contain $A$.
If $A$ is reducible, then $A$ is the union of two distinct lines $L_1$, $L_2$ and exactly three
points lie on each line, say $P_1, P_2, P_3$ lie on $L_1$ and $P_4, P_5, P_6$ lie on $L_2$.
Take $B$ to be the union of the lines $P_1 P_4$, $P_2 P_5$, $P_3 P_6$.
Choose $\f_{12} \in \SS^2 V^*$, $\f_{22} \in \SS^3 V^*$, $f \in \SS^6 V^*$ defining $A$, $B$, $C$.
We can find $\f_{11} \in \SS^3 V^*$, $\f_{21} \in \SS^4 V^*$
such that $f = \f_{11} \f_{22} - \f_{12} \f_{21}$.
Thus $(\f_{ij})$ represents a morphism in $W_{40}$ whose cokernel is isomorphic
to $\O_C(3)(-P_1 - \cdots - P_6)$.

To prove the inclusion $X_4 \subset \overline{X}_3$ fix a generic sheaf $\F$ in $X_4$
as in the proposition. Clearly we can find six distinct points $Q_i$ on $C$ that are not contained
in a conic curve such that $Q_i$ converges to $P_i$ for $1 \le i \le 6$.
According to \ref{4.4}, $\O_C(3)(-Q_1 - \cdots - Q_6)$ gives a point in $X_3$.
This point converges to $\F$ as $Q_i$ approach $P_i$.
By duality $X_4$ is also contained in the closure of $X_3^\D$.
\end{proof}


\section{The codimension $6$ stratum}

\begin{prop}
\label{6.1}
The sheaves $\F$ giving points in $\M(6,3)$ and satisfying the cohomological conditions
\[
\h^0(\F(-1))=1, \qquad \h^1(\F)=1, \qquad \h^0(\F \tensor \Om^1(1))=4,
\]
are precisely the sheaves having resolution of the form
\[
0 \lra \O(-3) \oplus \O(-2) \oplus \O(-1) \stackrel{\f}{\lra} \O(-1) \oplus \O \oplus \O(1) \lra \F \lra 0,
\]
\[
\f = \left[
\ba{ccc}
q_1 & \ell_1 & 0 \\
\star & \star & \ell_2 \\
\star & \star & q_2
\ea
\right],
\]
where $\ell_1, \ell_2 \in V^*$ are different from zero, $\ell_1$ does not divide $q_1$,
$\ell_2$ does not divide $q_2$.
\end{prop}

\begin{proof}
Let $\F$ give a point in $\M(6,3)$ and satisfy the cohomological conditions from the proposition.
At \ref{5.1} we found the resolution
\[
0 \lra \O(-3) \oplus \O(-2) \oplus \O(-1) \oplus \Om^1(1) \stackrel{\f}{\lra} \O(-1) \oplus 4\O \lra \F \lra 0,
\]
with $\f_{13}=0$, $\f_{14}=0$. The cokernel of $\f_{24}$ is isomorphic to $\O \oplus \O(1)$,
otherwise $\F$ would have a destabilising quotient sheaf that is the cokernel of a morphism
\[
\O(-3) \oplus \O(-2) \oplus \O(-1) \lra \O(-1) \oplus 2\O
\]
(see 2.1.4 \cite{mult_five}). We obtain a resolution as in the proposition.
The conditions on $\ell_1, \ell_2, q_1, q_2$ follow from the semi-stability of $\F$.
For instance, if $\ell_1$ divided $q_1$, then $\F$ would have a destabilising quotient
sheaf of the form $\O_L(-1)$ for a line $L \subset \P^2$.

Conversely, assume that $\F = \Coker(\f)$ for some morphism $\f$ as in the proposition.
Let $Z \subset \P^2$ be the zero-dimensional scheme of length $2$ given by the ideal
$(q_1, \ell_1)$. From the snake lemma we get an extension
\[
0 \lra \E \lra \F \lra \O_Z \lra 0,
\]
where $\E$ has a resolution
\[
0 \lra \O(-4) \oplus \O(-1) \stackrel{\psi}{\lra} \O \oplus \O(1) \lra \E \lra 0
\]
in which $\psi_{12}= \ell_2$, $\psi_{22}=q_2$.
According to 6.2 \cite{mult_six_one}, $\E$ gives a point in $\M(6,1)$.
Assume that there is a destabilising subsheaf $\F' \subset \F$.
We may assume that $\F'$ is stable and that it has multiplicity at most $5$.
Thus $\E \cap \F'$ is a proper subsheaf of $\E$, forcing $\pp(\E \cap \F') \le 0$.
It follows that $\F'$ gives a point in one of the following moduli spaces:
$\M(1,1)$, $\M(1,2)$, $\M(2,2)$, $\M(3,2)$.
In the first case we have a commutative diagram
\[
\xymatrix
{
0 \ar[r] & \O(-1) \ar[r] \ar[d]^-{\b} & \O \ar[r] \ar[d]^-{\a} & \F' \ar[r] \ar[d] & 0 \\
0 \ar[r] & \O(-3) \oplus \O(-2) \oplus \O(-1) \ar[r] & \O(-1) \oplus \O \oplus \O(1) \ar[r] & \F \ar[r] & 0
}
\]
in which $\a$ is injective because it is injective on global sections.
Thus
\[
\b \sim \left[
\ba{c}
0 \\ 0 \\ 1
\ea
\right] \qquad \text{and} \qquad \a \sim \left[
\ba{c}
0 \\ 1 \\ 0
\ea
\right] \quad \text{or} \quad \a \sim \left[
\ba{c}
0 \\ 0 \\ u
\ea
\right].
\]
Thus $\ell_2 = 0$ or $\ell_2$ divides $q_2$, contradicting our hypothesis.
In the second case we have a similar diagram in which $\b$ must be zero, hence $\a$
must factor through $\F'$, which is absurd.
If $\F'$ gives a point in $\M(2,2)$, then we have a diagram
\[
\xymatrix
{
0 \ar[r] & 2\O(-1) \ar[r] \ar[d]^-{\b} & 2\O \ar[r] \ar[d]^-{\a} & \F' \ar[r] \ar[d] & 0 \\
0 \ar[r] & \O(-3) \oplus \O(-2) \oplus \O(-1) \ar[r] & \O(-1) \oplus \O \oplus \O(1) \ar[r] & \F \ar[r] & 0
}.
\]
Since $\b$ cannot be injective, we see that $\a$ is not injective, so we my assume that
\[
\a = \left[
\ba{cc}
0 & 0 \\
0 & 0 \\
u_1 & u_2
\ea
\right], \qquad \qquad \b = \left[
\ba{cc}
0 & 0 \\
0 & 0 \\
1 & 0
\ea
\right],
\]
for some linearly independent one-forms $u_1, u_2$.
We obtain the contradictory conclusion that $\ell_2 = 0$.
Assume, finally, that $\F'$ gives a point in $\M(3,2)$.
We have a diagram
\[
\xymatrix
{
0 \ar[r] & \O(-2) \oplus \O(-1) \ar[r] \ar[d]^-{\b} & 2\O \ar[r] \ar[d]^-{\a} & \F' \ar[r] \ar[d] & 0 \\
0 \ar[r] & \O(-3) \oplus \O(-2) \oplus \O(-1) \ar[r] & \O(-1) \oplus \O \oplus \O(1) \ar[r] & \F \ar[r] & 0
}
\]
in which either $\a$ and $\b$ are both injective or
\[
\a \sim \left[
\ba{cc}
0 & 0 \\
0 & 0 \\
u_1 & u_2
\ea
\right] \qquad \text{and} \qquad \b \sim \left[
\ba{cc}
0 & 0 \\
1 & 0 \\
0 & 0
\ea
\right] \quad \text{or} \quad \b \sim \left[
\ba{cc}
0 & 0 \\
0 & 0 \\
u & 0
\ea
\right].
\]
Thus $\ell_1= 0$ or $\ell_2= 0$, which yields a contradiction.
\end{proof}

\noi
Denote $\W_5 = \Hom(\O(-3) \oplus \O(-2) \oplus \O(-1), \O(-1) \oplus \O \oplus \O(1))$
and let $W_5 \subset \W_5$ be the subset of morphism satisfying the conditions of \ref{6.1}.
Denote by $G_5$ the canonical group acting by conjugation $\W_5$.
Let $X_5 \subset \M(6,3)$ be the image of $W_5$ under the  map
$\f \mapsto [\Coker(\f)]$.

\begin{prop}
\label{6.2}
There exists a geometric quotient of $W_5$ modulo $G_5$, which is a proper open subset
of a fibre bundle with base $\Hilb(2) \times \Hilb(2)$ and fibre $\P^{23}$.
Moreover, $W_5/G_5$ is isomorphic to $X_5$.
In particular, $X_5$ has codimension $6$.
\end{prop}

\begin{proof}
The construction of $W_5/G_5$ is nearly identical to the construction of the quotient
at 3.2.3 \cite{mult_five}. The set of pairs $(q_1, \ell_1)$ is acted upon by
$\Aut(\O(-3) \oplus \O(-2))$ and the quotient is $\Hilb(2)$.
Thus the base $\Hilb(2) \times \Hilb(2)$ accounts for the set of quadruples $(q_1, \ell_1, q_2, \ell_2)$
modulo the action of the appropriate group.
The fibre accounts for the space $\Hom(\O(-3) \oplus \O(-2), \O \oplus \O(1))$ modulo the
subspace of morphisms represented by matrices of the form
\[
\left[
\ba{c}
v_1 \\ v_2
\ea
\right] \left[
\ba{cc}
q_1 & \ell_1
\ea
\right] + \left[
\ba{c}
\ell_2 \\ q_2
\ea
\right] \left[
\ba{cc}
u_1 & u_2
\ea
\right], \qquad v_1, u_2 \in V^*, \quad v_2, u_1 \in \SS^2 V^*.
\]
The argument at \ref{7.1} below shows that the canonical bijective morphism
$W_5/G_5 \to X_5$ is an isomorphism.
\end{proof}

\begin{prop}
\label{6.3}
The sheaves giving points in $X_5$ are precisely the extension sheaves of the form
\[
0 \lra \E \lra \F \lra \O_Z \lra 0,
\]
that satisfy the conditions $\h^1(\F)=1$, $\h^1(\F(1))=0$.
Here $\E$ denotes an arbitrary sheaf giving a point in the stratum $X_5$ of $\M(6,1)$
(cf. 6.2 \cite{mult_six_one}) and $Z \subset \P^2$ is an arbitrary zero-dimensional scheme
of length $2$.

The generic sheaves in $X_5$ have the form $\O_C(2)(P_1 + Q_1 - P_2 - Q_2)$,
where $C \subset \P^2$ is a smooth sextic curve and $P_1, Q_1, P_2, Q_2$
are four distinct points on $C$.
In particular, $X_5$ is contained in the closure of $X_2$.
\end{prop}

\begin{proof}
If $\F$ gives a point in $X_5$ then, clearly, $\h^1(\F)=1$ and $\h^1(\F(1))=0$.
Moreover, we saw at \ref{6.1} that $\F$ is an extension of $\O_Z$ by $\E$.
For the converse we combine the resolution of $\E$ at \ref{6.1} with the resolution
\[
0 \lra \O(-4) \lra \O(-3) \oplus \O(-2) \lra \O(-1) \lra \O_Z \lra 0.
\]
In order to do this we need to show that there is a global section of $\F(1)$
that maps to a section of $\O_Z$ generating this sheaf as an $\O_{\P^2}$-module.
Since $\h^0(\E(1))=8$ and $\h^0(\F(1))=9$, there is a global section of $\F(1)$
mapping to a non-zero section $s$ of $\O_Z$.
Consider the extension
\[
0 \lra \C_{z_1} \lra \O_Z \lra \C_{z_2} \lra 0,
\]
where $z_1, z_2$ are closed points of $\P^2$, not necessarily distinct.
If $s$ maps to zero in $\C_{z_2}$, then $s$ generates $\C_{z_1}$.
Let $\F'$ be the preimage of $\C_{z_1}$ in $\F$.
We may combine the above resolution of $\E$ with the resolution
\[
0 \lra \O(-3) \lra 2\O(-2) \lra \O(-1) \lra \C_{z_1} \lra 0
\]
to obtain the resolution
\[
0 \lra \O(-3) \lra \O(-4) \oplus 2\O(-2) \oplus \O(-1) \lra \O(-1) \oplus \O \oplus \O(1) \lra \F' \lra 0.
\]
From this we get $\h^0(\F')=5$, which is absurd, because $\h^0(\F')$ cannot exceed $\h^0(\F)$.
Thus the image of $s$ in $\C_{z_2}$ is non-zero.
This proves that $s$ generates $\O_Z$ in the case when $z_1=z_2$.
If $z_1 \neq z_2$, we can revert the roles of $z_1$ and $z_2$ in the above argument
to show that $s$ also does not vanish at $z_1$. Thus $s$ generates $\O_Z$.
Applying the horseshoe lemma we obtain a resolution
\[
0 \lra \O(-4) \lra \O(-4) \oplus \O(-3) \oplus \O(-2) \oplus \O(-1) \lra \O(-1) \oplus \O \oplus \O(1)
\lra \F \lra 0
\]
in which $\O(-4)$ can be canceled, otherwise $\h^1(\F)$ would equal $3$.
Thus we obtain $\f \in \W_5$ such that $\F = \Coker(\f)$.
It is clear that $\f$ is in $W_5$.

Let $W_{50} \subset W_5$ denote the open subset of morphisms $\f$ satisfying
the following conditions: the curve $C$ given by the equation $\det(\f)=0$ is smooth;
$q_1$ and $\ell_1$ vanish at two distinct points $P_1, Q_1$; $q_2$ and $\ell_2$ vanish
at two distinct points $P_2, Q_2$, that are also distinct from $P_1, Q_1$.
Let $X_{50}$ be the image of $W_{50}$ in $\M(6,3)$.
Given $\f \in W_{50}$ we can apply the snake lemma to deduce that
$\Coker(\f) \isom \O_C(2)(P_1+Q_1-P_2-Q_2)$.
Conversely, we must show that any such sheaf gives a point in $X_{50}$.
According to 6.2 \cite{mult_six_one}, the sheaf $\E = \O_C(2)(-P_2-Q_2)$
gives a point in the stratum $X_5$ of $\M(6,1)$.
Let $Z$ be the union of $P_1$ and $Q_1$.
Our aim is to apply the horseshoe lemma to the extension
\[
0 \lra \E \lra \F \lra \O_Z \lra 0.
\]
For this we need to show that $\F(1)$ has a global section that does not vanish at $P_1$
and also does not vanish at $Q_1$.
Let $\e_1, \e_2 \in \H^0(\O_Z)^*$ be the linear forms of evaluation at $P_1, Q_1$.
Let $\d \colon \H^0(\O_Z) \to \H^1(\E(1))$ be the connecting homomorphism
associated to the above exact sequence tensored with $\O(1)$.
We must show that each $\e_i$ is not orthogonal to the kernel of $\d$,
that is each $\e_i$ is not in the image of the dual map $\d^*$.
To see this we argue as at 2.3.2 \cite{mult_five}.
By Serre duality $\d^*$ is the restriction homomorphism
\[
\xymatrix
{
\H^0(\O_C(-3)(P_2+Q_2) \tensor \omega_C) \egal[d] \ar[r]
&
\H^0((\O_C(-3)(P_2+Q_2) \tensor \omega_C)_{|Z}) \egal[d]
\\
\H^0(\O_C(P_2+Q_2)) \egal[d]
&
\H^0(\O_C(P_2+Q_2)_{|Z}) \egal[d]
\\
\H^0(\O_C)
&
\H^0({\O_C}_{|Z})
}
\]
Let $Y$ be the union of $P_2$ and $Q_2$. The identification
$\H^0(\O_C(P_2+Q_2)) = \H^0(\O_C) = \C$ follows from the fact that the connecting
homomorphism associated to the exact sequence
\[
0 \lra \O_C \lra \O_C(P_2+Q_2) \lra \O_Y \lra 0
\]
is injective. Indeed, the dual of this homomorphism is surjective.
By Serre duality this is equivalent to saying that the restriction morphism
\[
\H^0(\O_C(3)) = \H^0(\omega_C) \lra \H^0({\omega_C}_{|Y}) = \H^0(\O_C(3)_{|Y})
\]
is surjective. This is obvious: we can find a cubic form vanishing at any of the points of $Y$
and that does not vanish at the other point.

Given a regular function $f$ on $C$, $\d^*(f)$ is a non-zero multiple of $\e_1$ precisely if
$f$ vanishes at $Q_1$ and does not vanish at $P_1$.
Such a regular function on $C$ does not exist.
Thus $\e_1$ is not in the image of $\d^*$ and ditto for $\e_2$.
We are now in position to apply the horseshoe lemma.
We obtain a resolution
\[
0 \lra \O(-4) \lra \O(-4) \oplus \O(-3) \oplus \O(-2) \oplus \O(-1) \lra
\O(-1) \oplus \O \oplus \O(1) \lra \F \lra 0.
\]
If the morphism $\O(-4) \to \O(-4)$ were zero, then, arguing as at 2.3.2 \cite{mult_five},
we could deduce that $\F$ is a split extension of $\O_Z$ by $\E$, which would be absurd.
Thus we may cancel $\O(-4)$ to obtain $\f \in W_{50}$ such that $\F \isom \Coker(\f)$.

To prove the inclusion $X_5 \subset \overline{X}_2$ fix a generic sheaf
$\O_C(2)(P_1+Q_1-P_2-Q_2)$ in $X_5$. Choose points $R_1, R_2$ on $C$ such that
\[
\O_C(2)(P_1+Q_1+R_1-P_2-Q_2-R_2)
\]
is a generic sheaf in $X_2$, cf. \ref{4.2}.
Making $R_1$ converge to $R_2$ we obtain a sequence of points in $X_2$
that converges to the fixed point of $X_5$.
\end{proof}


\section{The union of the codimension $5$ and codimension $6$ strata}

\noi
The stratum $X_4$, as we saw above, is parametrised by an open subset inside $\W_4$
while $X_5$ is parametrised by an open subset inside the closed
subset of $\W_5$ given by the condition $\f_{13}=0$.
It is thus natural to ask whether the union $X = X_4 \cup X_5$ is parametrised by an open
subset of $\W_5$. Notice that $X$ is locally closed, being the set of points represented by stable
sheaves inside the locally closed subset of $\M(6,3)$ given by the conditions
$\h^0(\F(-1))=1$, $\h^1(\F)=1$.
Let $W \subset \W_5$ be the set of injective morphisms $\f$ for which $\Coker(\f)$ is stable.
Concretely, $W$ is the union of $W_5$ with the set of morphisms equivalent to morphisms
of the form
\[
\left[
\ba{cc}
0 & 1 \\
\psi & 0
\ea
\right], \qquad \psi \in W_4.
\]
In this section we will write $G=G_5$.
Clearly $W$ is open, $G$-invariant and $X$ is its image under the map $\f \mapsto [\Coker(\f)]$.
We will show that $X$ is a geometric quotient of $W$ by $G$.
Following 4.3 \cite{drezet-maican}, we will describe $X$ as an open subset of the blow-up
of a certain compactification of $X_4$ along a smooth subvariety.
This compactification is the tower of bundles from proposition \ref{5.3}.
The locally closed subset $X_5 \subset X$ is an open subset of the special divisor
of this blow-up.

\begin{prop}
\label{7.1}
There exists a geometric quotient of $W$ by $G$, which is isomorphic to the subset $X \subset \M(6,3)$
of stable sheaves satisfying the conditions $\h^0(\F(-1))=1$, $\h^1(\F)=1$.
\end{prop}

\begin{proof}
Due to the fact that all sheaves giving points in $X$ are stable, it is easy to see that the fibres
of the canonical morphism $W \to X$ are $G$-orbits.
Using the method of 3.1.6 \cite{drezet-maican}, we will show that this is a categorical quotient
map. Applying \cite{mumford}, remark (2), p. 5, we will deduce that $X$ is normal.
From \cite{popov-vinberg}, theorem 4.2,
we will conclude that the map $W \to X$ is a geometric quotient.

Given $[\F]$ in $X$ we will construct $\f \in W$ in the fibre of $[\F]$
by performing algebraic operations on the Beilinson spectral sequence I converging to $\F(1)$.
The display diagram for its first term $\EE^1$ reads:
\[
\xymatrix
{
\H^1(\F(-1)) \tensor \Om^2(2) \ar[r]^-{\f_1} & \H^1(\F) \tensor \Om^1(1) \ar[r]^-{\f_2} & \H^1(\F(1)) \tensor \O \\
\H^0(\F(-1)) \tensor \Om^2(2) \ar[r]^-{\f_3} & \H^0(\F) \tensor \Om^1(1) \ar[r]^-{\f_4} & \H^0(\F(1)) \tensor \O
}.
\]
According to proposition \ref{9.1} below, the group $\H^1(\F(1))$ vanishes.
The above diagram takes the form
\[
\xymatrix
{
4\O(-1) \ar[r]^-{\f_1} & \Om^1(1) &  0 \\
\O(-1) \ar[r]^-{\f_3} & 4\Om^1(1) \ar[r]^-{\f_4} & 9\O
}.
\]
Notice that $\F(1)$ maps surjectively to $\Coker(\f_1)$, hence, arguing as at 4.2 \cite{mult_six_one},
$\Coker(\f_1)=0$ and $\Ker(\f_1) \isom \O(-2) \oplus \O(-1)$.
The display diagram for $\EE^2(\F(1))$ becomes
\[
\xymatrix
{
\Ker(\f_1) \ar[rrd]^-{\f_5} & 0 & 0 \\
\Ker(\f_3) & \Ker(\f_4)/\Im(\f_3) & \Coker(\f_4)
}.
\]
Thus $\EE^3= \EE^{\infty}$, $\f_3$ and $\f_5$ are injective, $\F(1)$ is isomorphic to $\Coker(\f_5)$
and $\Ker(\f_4)= \Im(\f_3)$. Note that $\f_5$ factors through $9\O$, so we have the resolution
\[
0 \lra \O(-1) \lra \O(-2) \oplus \O(-1) \oplus 4\Om^1(1) \lra 9\O \lra \F(1) \lra 0.
\]
The morphism $4\Om^1(1) \to 9\O$ factors through the morphism $4\Om^1(1) \to 12\O$
obtained from the Euler sequence. This leads to the resolution
\[
0 \lra \O(-1) \lra \O(-2) \oplus \O(-1) \oplus 12\O \stackrel{\f}{\lra} 9\O \oplus 4\O(1) \lra \F(1) \lra 0.
\]
Notice that $\rank(\f_{13}) \ge 8$, otherwise $\F(1)$ would map surjectively to the cokernel of a morphism
$\O(-2) \oplus \O(-1) \to 2\O$, in violation of semi-stability.
Canceling $8\O$ and twisting by $-1$ we get a resolution
\[
0 \lra \O(-2) \stackrel{\psi}{\lra} \O(-3) \oplus \O(-2) \oplus 4\O(-1) \stackrel{\f}{\lra} \O(-1) \oplus 4\O \lra \F \lra 0
\]
in which $\psi_{11}=0$, $\psi_{21}=0$. If $\f_{13}=0$, then, arguing as in the proof of 2.1.4 \cite{mult_five},
we can show that $\Coker(\psi_{31}) \isom \O(-1) \oplus \Om^1(1)$. If $\f_{13}=0$ roughly the same
argument applies. We are led to a resolution
\[
0 \lra \O(-3) \oplus \O(-2) \oplus \O(-1) \oplus \Om^1(1) \stackrel{\f}{\lra} \O(-1) \oplus 4\O \lra \F \lra 0
\]
in which $\f_{14}=0$. As at loc.cit., we have the factorisation $\f_{24}= \b \circ i$.
If $\b$ were not injective, then $\F$ would map surjectively onto the cokernel of a morphism
\[
\O(-3) \oplus \O(-2) \oplus \O(-1) \lra \O(-1) \oplus 2\O,
\]
in violation of semi-stability. Thus $\b$ is injective, i.e. $\Coker(\f_{24})\isom \O \oplus \O(1)$,
and we get the resolution
\[
0 \lra \O(-3) \oplus \O(-2) \oplus \O(-1) \stackrel{\f}{\lra} \O(-1) \oplus \O \oplus \O(1) \lra \F \lra 0.
\]
By definition $\f$ belongs to $W$.
\end{proof}

\noi
In the sequel we will represent elements $\f \in \W_5$ and tangent vectors $w$ at $\f$ by matrices
\[
\f= \left[
\ba{ccc}
q_1 & \ell_1 & c \\
f_1 & q & \ell_2 \\
p & f_2 & q_2
\ea
\right], \qquad w = \left[
\ba{ccc}
q_1' & \ell_1' & c' \\
f_1' & q' & \ell_2' \\
p' & f_2' & q_2'
\ea
\right].
\]
Elements $\psi \in \W_4$ and tangent vectors $u$ at $\psi$ will be represented by matrices
\[
\psi= \left[
\ba{cc}
f_1 & q \\
p & f_2
\ea
\right], \qquad \qquad u = \left[
\ba{cc}
f_1' & q' \\
p' & f_2'
\ea
\right].
\]
Consider the open $G_4$-invariant subset $U \subset \W_4$ given by the conditions
$q \neq 0$, $q \nmid f_1$, $q \nmid f_2$.
Clearly $U$ is contained in the set $W_4' \setminus \Sigma$ of \ref{5.3},
hence there exists a geometric quotient $U/G_4$ as an open subset of the projective
variety $(W_4' \setminus \Sigma)/G_4$, which was described as a tower of bundles at loc.cit.
In particular, $U/G_4$ is smooth. Denote by $\g \colon U \to U/G_4$ the quotient map.
Note that $W_4$ is the subset of injective morphisms inside $U$.
Its complement $Z= U \setminus W_4$ consists of those morphisms represented by matrices of the form
\[
\left[
\ba{c}
\ell_2 \\ q_2
\ea
\right] \left[
\ba{cc}
q_1 & \ell_1
\ea
\right]
\]
where $\ell_1, \ell_2 \in V^*$ are non-zero, $q_1 \in \SS^2 V^*$ is non-divisible by $\ell_1$,
$q_2 \in \SS^2 V^*$ is non-divisible by $\ell_2$.
Thus
$
Z/G_4 \isom \Hilb(2) \times \Hilb(2).
$
This is a smooth subvariety of $U/G_4$ which we denote by $S$.
Let $B$ denote the blow-up of $U/G_4$ along $S$ and let $\b \colon B \to U/G_4$ denote the
blowing-down map.
Following 4.3 \cite{drezet-maican} we define the map
\[
\d \colon W \lra U, \qquad \d(\f)= c \left[
\ba{cc}
f_1 & q \\
p & f_2
\ea
\right] - \left[
\ba{c}
\ell_2 \\ q_2
\ea
\right] \left[
\ba{cc}
q_1 & \ell_1
\ea
\right].
\]
Notice that $\g \circ \d$ maps the smooth hypersurface $W_5$ to $S$.
By the universal property of the blow-up there is a morphism $\a$ making the diagram
commute:
\[
\xymatrix
{
W \ar[r]^-{\d} \ar[d]^-{\a} & U \ar[d]^-{\g} \\
B \ar[r]^-{\b} & U/G_4
}.
\]

\begin{prop}
\label{7.2}
The image of $\a$ is an open subset $B_0$ of $B$ and the map $W \to B_0$ is a geometric
quotient for the action of $G$. Thus $X$ is isomorphic to $B_0$.
\end{prop}

\begin{proof}
Firstly, we will show that the fibres of $\a$ are the $G$-orbits.
Secondly, we will show that $\a$ has surjective differential at every point.
Thus $B_0$ is open, hence smooth.
Applying \cite{popov-vinberg}, theorem 4.2, we will conclude that the map $W \to B_0$
is a geometric quotient modulo $G$.

Notice that $\d(W \setminus W_5)= W_4$. In fact, it is a trivial observation that the fibres
of the composite map $W \setminus W_5 \stackrel{\d}{\to} W_4 \stackrel{\g}{\to} W_4/G_4$
are the $G$-orbits. Thus $\a \colon W \setminus W_5 \to B \setminus E$ is a geometric quotient,
$E$ being the exceptional divisor of the blow-up.

Fix a point $s \in S$. Recall that $\b^{-1}(s) = \P(\NN_s)$, where $\NN$ is the normal bundle of $S$
in $U/G_4$. Choose a point $\psi \in U$ lying over $s$ of the form
\[
\psi = \left[
\ba{c}
\ell_2 \\ q_2
\ea
\right] \left[
\ba{cc}
q_1 & \ell_1
\ea
\right].
\]
We identify $\NN_s$ with the fibre over $\psi$ of the normal bundle of $Z$ in $U$,
that is with $\TTT_{\psi} U/\TTT_{\psi} Z$.
Since $Z$ is smooth, $\TTT_{\psi} Z$ can be identified with the space of matrices
\[
\left[
\ba{c}
\ell_2 \\ q_2
\ea
\right] \left[
\ba{cc}
q_1' & \ell_1'
\ea
\right] + \left[
\ba{c}
\ell_2' \\ q_2'
\ea
\right] \left[
\ba{cc}
q_1 & \ell_1
\ea
\right], \qquad \ell_1', \ell_2' \in V^*, \quad q_1', q_2' \in \SS^2 V^*.
\]
Choose $\f \in W_5$ lying over $\psi$. The differential of $\d$ at $\f$ is given by the formula
\[
\dd \d_{\f} (w) = c' \left[
\ba{cc}
f_1 & q \\
p & f_2
\ea
\right] - 
\left[
\ba{c}
\ell_2 \\ q_2
\ea
\right] \left[
\ba{cc}
q_1' & \ell_1'
\ea
\right] -
\left[
\ba{c}
\ell_2' \\ q_2'
\ea
\right] \left[
\ba{cc}
q_1 & \ell_1
\ea
\right].
\]
The tangent vector $\nu \in \TTT_{\f} W$ given by the matrix
\[
\left[
\ba{ccc}
0 & 0 & 1 \\
0 & 0 & 0 \\
0 & 0 & 0
\ea
\right]
\]
lies in the normal direction to $W_5$ at $\f$, hence
\[
\a(\f) = ( \langle \dd \d_{\f} (\nu) \mod \TTT_{\psi} Z \rangle, s)
= \left( \langle \left[
\ba{cc}
f_1 & q \\
p & f_2
\ea
\right] \mod \TTT_{\psi} Z \rangle, s \right).
\]
Fix $g \in G$. Denote by the same letter the map $W \to W$ of multiplication by $g$.
Since $W_5$ is $G$-invariant, $\dd g_{\f}(\nu)$ is a normal vector to $W_5$ at $g \f$.
Thus, taking into account that $\g \circ \d \circ g = \g \circ \d$, we have
\begin{align*}
\a(g \f) & = ( \langle \dd (\g \circ \d)_{g \f}(\dd g_{\f}(\nu)) \mod \TTT_{\g \circ \d (g \f)} S \rangle, \g \circ \d (g\f)) \\
& = ( \langle \dd (\g \circ \d \circ g)_{\f}(\nu) \mod \TTT_{\g \circ \d \circ g (\f)} S \rangle, \g \circ \d \circ g (\f)) \\
& = ( \langle \dd (\g \circ \d)_{\f}(\nu) \mod \TTT_{\g \circ \d (\f)} S \rangle, \g \circ \d (\f)) \\
& = \a(\f).
\end{align*}
This shows that $\a$ is $G$-equivariant.
Assume now that $\a(\bar{\f})= \a(\f)$ for $\bar{\f}, \f \in W_5$.
Performing, possibly, elementary operations on $\bar{\f}$ we may assume that
$\d (\bar{\f}) = \d (\f)$, i.e.
\[
\left[
\ba{c}
\bar{\ell}_2 \\ \bar{q}_2
\ea
\right] \left[
\ba{cc}
\bar{q}_1 & \bar{\ell}_1
\ea
\right] = \left[
\ba{c}
\ell_2 \\ q_2
\ea
\right] \left[
\ba{cc}
q_1 & \ell_1
\ea
\right].
\]
Thus $\ell_1$ divides both  $\bar{\ell}_1 \bar{\ell}_2$ and $\bar{\ell}_1 \bar{q}_2$,
hence $\bar{\ell}_1 = a \ell_1$ for some $a \in \C^*$, hence
$\bar{q}_1 = a q_1$, $\bar{\ell}_2 = a^{-1}\ell_2$, $\bar{q}_2 = a^{-1} q_2$.
Performing, possibly, elementary operations on $\bar{\f}$, we may assume that
\[
\bar{\f} = \left[
\ba{ccc}
q_1 & \ell_1 & 0 \\
\bar{f}_1 & \bar{q} & \ell_2 \\
\bar{p} & \bar{f}_2 & q_2
\ea
\right].
\]
The relation $\a(\bar{\f}) = \a(\f)$ is equivalent to saying that
\[
\langle \left[
\ba{cc}
\bar{f}_1 & \bar{q} \\
\bar{p} & \bar{f}_2
\ea
\right] \mod \TTT_{\psi} Z \rangle =
\langle \left[
\ba{cc}
f_1 & q \\
p & f_2
\ea
\right] \mod \TTT_{\psi} Z \rangle,
\]
which is equivalent to saying that there are $b \in \C^*$ and $\ell_1', \ell_2' \in V^*$,
$q_1', q_2' \in \SS^2 V^*$ such that
\[
\left[
\ba{cc}
\bar{f}_1 & \bar{q} \\
\bar{p} & \bar{f}_2
\ea
\right] = b \left[
\ba{cc}
f_1 & q \\
p & f_2
\ea
\right] + \left[
\ba{c}
\ell_2 \\ q_2
\ea
\right] \left[
\ba{cc}
q_1' & \ell_1'
\ea
\right] + \left[
\ba{c}
\ell_2' \\ q_2'
\ea
\right] \left[
\ba{cc}
q_1 & \ell_1
\ea
\right].
\]
From this it immediately follows that $\bar{\f}$ is in the orbit of $\f$.
This proves that the fibres of $\a$ are the $G$-orbits.

It remains to show that $\a$ has surjective differential at every point $\f \in W_5$.
By construction $\dd \a_{\f}(\nu)$ is a normal vector to $E$ at $\a(\f)$.
Thus we only need to show that the restriction $W_5 \stackrel{\a}{\to} E$ has surjective
differential at $\f$.
The composite map $W_5 \stackrel{\a}{\to} E \stackrel{\b}{\to} S$ has surjective differential
at every point because the maps $W_5 \stackrel{\d}{\to} Z$ and $Z \stackrel{\g}{\to} S$
enjoy this property.
Denote $s = \g \circ \d (\f)$, $\psi= \d(\f)$.
We have reduced the problem to showing that the map
$\a^{-1} \b^{-1}(s) \stackrel{\a}{\to} \P(\NN_s)$ has surjective differential at $\f$.
Denote by $A$ the affine subset of $\W_5$ represented by matrices of the form
\[
\left[
\ba{ccc}
q_1 & \ell_1 & 0 \\
\bar{f}_1 & \bar{q} & \ell_2 \\
\bar{p} & \bar{f}_2 & q_2
\ea
\right],
\]
$\bar{q} \in \SS^2 V^*$, $\bar{f}_1, \bar{f}_2 \in \SS^3 V^*$, $\bar{p} \in \SS^4 V^*$.
We have a commutative diagram
\[
\xymatrix
{
A \cap W_5 \ar[r]^-{\pi} \ar[d]^-{i} & \NN_s \setminus \{ 0 \} \ar[d] \\
\a^{-1} \b^{-1}(s) \ar[r]^-{\a} & \P(\NN_s)
},
\]
where $i$ is the inclusion map and $\pi$ is given by the formula
\[
\left[
\ba{ccc}
q_1 & \ell_1 & 0 \\
\bar{f}_1 & \bar{q} & \ell_2 \\
\bar{p} & \bar{f}_2 & q_2
\ea
\right] \longmapsto 
\left[
\ba{cc}
\bar{f}_1 & \bar{q} \\
\bar{p} & \bar{f}_2
\ea
\right] \mod \TTT_{\psi} Z.
\]
Thus $\pi$ is the restriction of  linear surjective map $A \to \NN_s$ defined by the same
formula as above. As such, $\pi$ has surjective differential at every point.
Thus $\a \circ i$ has surjective differential at every point.
We conclude that the map $\a^{-1} \b^{-1} (s) \stackrel{\a}{\to} \P(\NN_s)$
has surjective differential at $\f$.
\end{proof}

\noi
In the next proposition we give some information about the closed subvariety
$E_1 = B \setminus B_0$ of $B$. Clearly $E_1$ is a subvariety of the exceptional divisor.
Let $\Delta_1 \subset \Hilb(2) \times \Hilb(2)$ denote the subvariety of pairs
of zero-dimensional subschemes of length $2$ that have at least one point in common.
Let $\Delta \subset \Hilb(2) \times \Hilb(2)$ denote the diagonal subvariety.

\begin{prop}
\label{7.3}
The blowing-down map sends $E_1$ to $\Delta_1$.
Its restriction $\b_1 \colon E_1 \to \Delta_1$ is an isomorphism away from $\Delta$.
The fibres of $\b_1$ over points of $\Delta$ are isomorphic to $\P^1$.
\end{prop}

\begin{proof}
We study the fibres of $\b_1$ by explicit calculations, as in 4.3.4 \cite{drezet-maican}.
Consider first a point $s = \g(\psi) \in S$,
\[
\psi = \left[
\ba{c}
\ell_2 \\ q_2
\ea
\right] \left[
\ba{cc}
q_1 & \ell_1
\ea
\right],
\]
where $\ell_2$ is not a multiple of $\ell_1$.
We choose the following canonical form for elements $\f \in W_5$ mapping to $s$:
\[
\left[
\ba{ccc}
q_1(Y,Z) & X & 0 \\
f_1(X,Z) & a Z^2 & Y \\
p(X,Y,Z) & f_2(Y,Z) & q_2(X,Z)
\ea
\right],
\]
where $a \in \C$. Let $A$ denote the vector space of matrices of the form
\[
\left[
\ba{cc}
f_1(X,Z) & a Z^2 \\
p(X,Y,Z) & f_2(Y,Z)
\ea
\right].
\]
The map $\a \colon A \setminus \{ 0 \} \to \P(\NN_s)$ given by the formula
\[
\left[
\ba{cc}
f_1 & q \\
p & f_2
\ea
\right] \longmapsto \langle \left[
\ba{cc}
f_1 & q \\
p & f_2
\ea
\right] \mod \TTT_{\psi} Z \rangle
\]
is well-defined and clearly surjective. Indeed, assume that
\[
\left[
\ba{cc}
f_1 & q \\
p & f_2
\ea
\right] = \left[
\ba{c}
Y \\ q_2
\ea
\right] \left[
\ba{cc}
q_1' & \ell_1'
\ea
\right] + \left[
\ba{c}
\ell_2' \\ q_2'
\ea
\right] \left[
\ba{cc}
q_1 & X
\ea
\right]
\]
for some $\ell_1', \ell_2' \in V^*$, $q_1', q_2' \in \SS^2 V^*$.
The relation $aZ^2 = Y \ell_1' + X \ell_2'$ shows that $a=0$ and $\ell_1'= bX$,
$\ell_2'=-bY$ for some $b \in \C$.
The relation $f_1 = Y q_1' - bYq_1$ shows that $f_1=0$ and $q_1'=b q_1$.
Analogously $f_2=0$ and $q_2' = -b q_2$. We have
\[
p = q_2 q_1' + q_2' q_1 = q_2 b q_1 - b q_2 q_1 =0.
\]
Thus $A \cap \TTT_{\psi}Z = \{ 0 \}$.
Note that $E_1 \cap \P(\NN_s)$ is the image under $\a$ of the subset of $A \setminus \{ 0 \}$
given by the condition $\det(\f) \neq 0$.
This condition reads
\[
a Z^2 q_1 q_2 - q_1 Y f_2 - X f_1 q_2 + XY p =0.
\]
There are linear forms $\l_1(Y,Z)$, $\l_2(X,Z)$ and constants $a_1$, $a_2$ such that
\[
q_1 = Y \l_1(Y,Z) + a_1 Z^2, \qquad \qquad q_2 = X \l_2(X,Z) + a_2 Z^2.
\]
Notice that $s$ belongs to $\Delta_1$ precisely if $a_1=0$ and $a_2=0$.
The above condition becomes
\[
q_1(a Z^2 X \l_2 + a a_2 Z^4 - Y f_2)= X f_1 q_2 - XY p.
\]
Since $X$ divides $a a_2 Z^4 - Y f_2$ we have $a a_2=0$ and $f_2 =0$.
Analogously we have $a a_1 =0$ and $f_1 =0$.
The condition $\det(\f)=0$ becomes
\[
a Z^2(Y\l_1 + a_1 Z^2)(X\l_2 + a_2 Z^2) + XYp =0, \quad \text{that is} \quad
p= - a Z^2 \l_1 \l_2.
\]
If $s \notin \Delta_1$, then $a=0$, hence $E_1 \cap \P(\NN_s)= \emptyset$.
If $s \in \Delta_1$, then $E_1 \cap \P(\NN_s)$ is a point, namely it is the image under $\a$
of the set of non-zero matrices of the form
\[
\left[
\ba{cc}
0 & a Z^2 \\
-a Z^2 \l_1 \l_2 & 0
\ea
\right].
\]
Assume now that $s$ is such that $\ell_2$ is a multiple of $\ell_1$.
The elements $\f \in W_5$ mapping to $s$ have canonical form
\[
\left[
\ba{ccc}
q_1(Y,Z) & X & 0 \\
f_1(Y,Z) & q(Y,Z) & X \\
p(X,Y,Z) & f_2(Y,Z) & q_2(Y,Z)
\ea
\right].
\]
Let $A$ denote the vector space of matrices of the form
\[
\left[
\ba{cc}
f_1(Y,Z) & q(Y,Z) \\
p(X,Y,Z) & f_2(Y,Z)
\ea
\right].
\]
The map $\a \colon A \setminus \TTT_{\psi} Z \to \P(\NN_s)$ defined by the same formula as above
is clearly surjective. We claim that $A \cap \TTT_{\psi} Z$ is the subspace of matrices of the form
\[
\left[
\ba{cc}
- \ell q_1 & 0 \\
0 & \ell q_2
\ea
\right],
\]
where $\ell$ is a linear form in $Y$ and $Z$.
Indeed, $A \cap \TTT_{\psi} Z$ is given by the relation
\[
\left[
\ba{cc}
f_1 & q \\
p & f_2
\ea
\right] = \left[
\ba{c}
X \\ q_2
\ea
\right] \left[
\ba{cc}
q_1' & \ell_1'
\ea
\right] + \left[
\ba{c}
\ell_2' \\ q_2'
\ea
\right] \left[
\ba{cc}
q_1 & X
\ea
\right].
\]
The relation $q= X(\ell_1'+ \ell_2')$ forces $q=0$ and $\ell_1' = -\ell_2'= aX+bY+cZ$.
The relation $f_1 = Xq_1'- (aX+bY+cZ)q_1$ shows that
\[
f_1 = -(bY+cZ)q_1, \quad q_1'= a q_1.
\quad \text{Analogously we have} \quad
f_2 = (bY+cZ) q_2, \quad q_2'= - a q_2,
\]
hence
\[
p = q_2 q_1' + q_2' q_1 = q_2 a q_1 - a q_2 q_1 =0.
\]
Next we determine the subspace of $A$ given by the relation $\det(A)=0$, which reads
\[
q_1 q q_2 - q_1 X f_2 - f_1 X q_2 + X^2 p = 0.
\]
Since $q_1 q q_2$ is divisible by $X$ we see that $q=0$ and $X p= q_1 f_2 + f_1 q_2$,
forcing $p=0$ and $f_1 q_2 = - f_2 q_1$.
If $s \notin \Delta_1$, then $q_1$ and $q_2$ have no common factor, hence
\[
\left[
\ba{cc}
f_1 & q \\
p & f_2
\ea
\right]
\]
belongs to $A \cap \TTT_{\psi} Z$. In this case we get $E_1 \cap \P(\NN_s)= \emptyset$.
If $s \in \Delta_1 \setminus \Delta$, then the greatest common divisor of $q_1$ and $q_2$
is a linear form, so the condition $\det(\f)=0$ determines a three-dimensional subspace 
of $A$. In this case $E_1 \cap \P(\NN_s)$ is a point. If $s \in \Delta$, then $q_2$ is a multiple of
$q_1$ and we get a subspace of dimension $4$, hence $\E_1 \cap \P(\NN_s)$ is isomorphic
to $\P^1$.

In conclusion, $\b(E_1)= \Delta_1$, the restriction of the blow-down map $\b_1 \colon E_1 \to \Delta_1$
is bijective over $\Delta_1 \setminus \Delta$ and its fibres over $\Delta$ are isomorphic to $\P^1$.
Since $\Delta_1 \setminus \Delta$ is smooth, $\b_1$ determines an isomorphism
$E_1 \setminus \b^{-1} (\Delta) \to \Delta_1 \setminus \Delta$.
\end{proof}

\noi
Notice that $\Delta$ is a smooth subvariety of $\Delta_1$ of codimension $2$.
It is thus natural to ask whether $E_1$ is the blow-up of $\Delta_1$ along $\Delta$.


\section{The codimension $8$ stratum}

\begin{prop}
\label{8.1}
Let $\F$ give a point in $\M(6,3)$ and satisfy the conditions $\h^0(\F(-1))=2$, $\h^1(\F)=2$.
Then $\h^0(\F \tensor \Om^1(1))=6$. These sheaves are precisely the sheaves having
resolution of the form
\[
0 \lra 2\O(-3) \oplus \O \stackrel{\f}{\lra} \O(-2) \oplus 2\O(1) \lra \F \lra 0,
\]
where $\f_{11}$ has linearly independent entries and ditto for $\f_{22}$.
\end{prop}

\begin{proof}
Assume that $\F$ gives a point in $\M(6,3)$ and satisfies the conditions
$\h^0(\F(-1))=2$, $\h^1(\F)=2$. Put $m= \h^0(\F \tensor \Om^1(1))$.
The Beilinson free monad for $\F$
\[
0 \lra 2\O(-2) \lra 5\O(-2) \oplus m\O(-1) \lra m\O(-1) \oplus 5\O \stackrel{\eta}{\lra} 2\O \lra 0
\]
yields the resolution
\[
0 \lra 2\O(-2) \stackrel{\psi}{\lra} 5\O(-2) \oplus m\O(-1) \stackrel{\f}{\lra} \Ker(\eta_{11}) \oplus 5\O \lra \F
\lra 0.
\]
Here $\psi_{11}=0$, $\f_{12}=0$. We have $m \le 7$ because $\F$ maps surjectively to $\Coker(\f_{11})$.
If $m=7$, then $\Coker(\f_{11})$ would be a destabilising quotient sheaf of $\F$.
Thus $m \le 6$. The cases when $m \le 5$ can be eliminated as in the proof of 3.1.3 \cite{mult_five}.
This proves that $m=6$.
Arguing as at 3.2.5 op.cit., we can show that $\Coker(\psi_{12})\isom 2\Om^1(1)$.
According to \cite{maican-duality}, lemma 3, dualising the free monad for $\F$ yields a monad
for the dual sheaf $\F^\D$. The latter gives a point in $\M(6,-3)$, cf. op.cit. From what was said
above it follows that the morphism $\eta_{11}^\T \in \Hom(2\O(-3), 6\O(-2))$ has cokernel $2\Om^1$,
which is equivalent to saying that $\Ker(\eta_{11}) \isom 2\Om^1$.
Presently we arrive at the resolution
\[
0 \lra 5\O(-2) \oplus 2\Om^1(1) \stackrel{\f}{\lra} 2\Om^1 \oplus 5\O \lra \F \lra 0.
\]
Using the Euler sequence we get a resolution
\[
0 \lra 2\O(-3) \oplus 5\O(-2) \oplus 2\Om^1(1) \stackrel{\f}{\lra} 6\O(-2) \oplus 5\O \lra \F \lra 0
\]
in which $\f_{13}=0$. If $\rank(\f_{12}) \le 4$, then $\F$ would map surjectively onto the cokernel
of a morphism $2\O(-3) \to 2\O(-2)$, in violation of semi-stability.
Thus $\rank(\f_{12})=5$ and canceling $5\O(-2)$ we get the resolution
\[
0 \lra 2\O(-3) \oplus 2\Om^1(1) \lra \O(-2) \oplus 5\O \lra \F \lra 0.
\]
Using again the Euler sequence we get the resolution
\[
0 \lra 2\O(-3) \oplus 6\O \stackrel{\f}{\lra} \O(-2) \oplus 5\O \oplus 2\O(1) \lra \F \lra 0.
\]
If $\rank(\f_{22}) \le 4$, then $\F$ would map surjectively to the
cokernel of a morphism $2\O(-3) \to \O(-2) \oplus \O$, in violation of semi-stability.
Thus $\rank(\f_{22})=5$ and canceling $5\O$ we get the resolution
\[
0 \lra 2\O(-3) \oplus \O \stackrel{\f}{\lra} \O(-2) \oplus 2\O(1) \lra \F \lra 0.
\]
The conditions on $\f_{11}$ and on $\f_{22}$ from the proposition follow from the semi-stability
of $\F$.

Conversely, assume that $\F$ has a resolution as in the proposition.
From the snake lemma we get an extension
\[
0 \lra \E \lra \F \lra \C_x \lra 0,
\]
where $\C_x = \Coker(\f_{11})$ and $\E$ has a resolution
\[
0 \lra \O(-4) \oplus \O \stackrel{\psi}{\lra} 2\O(1) \lra \E \lra 0
\]
in which $\psi_{12}= \f_{22}$.
According to 6.1 \cite{mult_six_two}, $\E$ is stable.
It is now straightforward to check that any possibly destabilising subsheaf of $\F$
must be isomorphic to $\O_L$ for some line $L \subset \P^2$.
Assume that $\F$ had such a subsheaf. We would then get a commutative diagram
\[
\xymatrix
{
0 \ar[r] & \O(-1) \ar[r]^-{\ell} \ar[d]^-{\b} & \O \ar[r] \ar[d]^-{\a} & \O_L \ar[r] \ar[d] & 0 \\
0 \ar[r] & 2\O(-3) \oplus \O \ar[r]^-{\f} & \O(-2) \oplus 2\O(1) \ar[r] & \F \ar[r] & 0
}
\]
with injective $\a$ and $\b$. The relation $\f_{22} \b_{21} = \a_{21} \ell$ shows that
$\b_{21}$ is a multiple of $\ell$ and that $\a_{21}$ is a multiple of $\f_{22}$.
Thus $\Coker(\a)$ is torsion-free. Since $\Coker(\b)$ maps injectively to $\Coker(\a)$,
it follows that $\Coker(\b)$ is also torsion-free. This is absurd because $\O_L$ is a direct
summand of $\Coker(\b)$.
\end{proof}

\noi
Let $\W_6 = \Hom(2\O(-3) \oplus \O, \O(-2) \oplus 2\O(1))$ and let $W_6 \subset \W_6$
be the open subset of morphisms $\f$ as in proposition \ref{8.1}.
Let
\[
G_6 = (\Aut(2\O(-3) \oplus \O) \times \Aut(\O(-2) \oplus 2\O(1)))/\C^*
\]
be the natural group acting by conjugation on $\W_6$.
Let $X_6 \subset \M(6,3)$ be the image of $W_6$ under the canonical morphism
$\f \mapsto [\Coker(\f)]$.

\begin{prop}
\label{8.2}
There exists a geometric quotient of $W_6$ by $G_6$, which is a proper open subset of a
fibre bundle over $\P^2 \times \P^2$ with fibre $\P^{25}$.
Moreover, $W_6/G_6$ is isomorphic to $X_6$.
In particular, $X_6$ has codimension $8$.
\end{prop}

\begin{proof}
The construction of $W_6/G_6$ is nearly the same as the construction of the quotient at 2.2.4 \cite{mult_five}.
We consider the open subset $W_6' \subset \W_6$ given by the following conditions:
$\f_{11}$ has linearly independent entries, $\f_{22}$ has linearly independent entries and
\[
\f_{21} \neq v \f_{11} + \f_{22} u \qquad \text{for any} \quad u \in \Hom(2\O(-3),\O), \quad v \in \Hom(\O(-2),2\O(1)).
\]
There exists a geometric quotient $W_6'/G_6$, which is a fibre bundle over $\P^2 \times \P^2$
with fibre $\P^{25}$. The quotient map takes $\f$ to $((x,y), \langle \f_{21} \rangle)$,
where $x$ is the common zero of the entries of $\f_{11}$, $y$ is the common zero of the entries of $\f_{22}$
and $\langle \f_{21} \rangle$ denotes the line spanned by the image of $\f_{21}$ in the cokernel of the
canonical morphism
\begin{multline*}
\Hom(2\O(-3), \O(-2)) \tensor \Hom(\O(-2), 2\O(1)) \oplus \Hom(2\O(-3), \O) \tensor \Hom(\O, 2\O(1)) \\
\lra \Hom(2\O(-3), 2\O(1)).
\end{multline*}
The quotient $W_6/G_6$ is a proper open subset of the projective variety $W_6'/G_6$.

Fix $\F$ in $X_6$. The first term of the Beilinson spectral sequence II converging to $\F$ has display diagram
\[
\xymatrix
{
5\O(-2) \ar[r]^-{\f_1} & 6\O(-1) \ar[r]^-{\f_2} & 2\O \\
2\O(-2) \ar[r]^-{\f_3} & 6\O(-1) \ar[r]^-{\f_4} & 5\O
}.
\]
Arguing as at 3.2.5 \cite{mult_five}, we can show that $\Coker(\f_3) \isom 2\Om^1(1)$.
The sheaf $\F^\D(1)$ also gives a point in $X_6$ and the associated Beilinson spectral sequence
has display diagram
\[
\xymatrix
{
5\O(-2) \ar[r]^-{\f_4^\T} & 6\O(-1) \ar[r]^-{\f_3^\T} & 2\O \\
2\O(-2) \ar[r]^-{\f_2^\T} & 6\O(-1) \ar[r]^-{\f_1^\T} & 5\O
}.
\]
Thus $\Coker(\f_2^\T) \isom 2\Om^1(1)$, which is equivalent to saying that $\Ker(\f_2) \isom 2\Om^1$.
Denote $\CC = \Ker(\f_2)/\Im(\f_1)$. We have an exact sequence
\[
0 \lra \Ker(\f_1) \lra 2\O(-3) \oplus 5\O(-2) \stackrel{\xi}{\lra} 6\O(-2) \lra \CC \lra 0.
\]
Notice that $\rank(\xi_{12})=5$, otherwise $\F$ would map surjectively onto the cokernel of a morphism
$2\O(-3) \to 2\O(-2)$, in violation of semi-stability.
It is clear now that $\CC$ is isomorphic to the structure sheaf $\C_x$ of a closed point $x \in \P^2$
and $\Ker(\f_1) \isom \O(-4)$.
We have the exact sequence
\[
0 \lra 2\Om^1(1) \lra 5\O \lra \Coker(\f_4) \lra 0,
\]
which leads to an exact sequence
\[
0 \lra 6\O \lra 5\O \oplus 2\O(1) \lra \Coker(\f_4) \lra 0.
\]
The exact sequence (2.2.5) \cite{drezet-maican} reads
\[
0 \lra \O(-4) \stackrel{\f_5}{\lra} \Coker(\f_4) \lra \F \lra \C_x \lra 0.
\]
Put $\E= \Coker(\f_5)$. Clearly $\f_5$ factors through $5\O \oplus 2\O(1)$,
hence we have a resolution
\[
0 \lra \O(-4) \oplus 6\O \stackrel{\psi}{\lra} 5\O \oplus 2\O(1) \lra \E \lra 0.
\]
We have $\rank(\psi_{12})=5$, otherwise $\E$, hence also $\F$, would have a subsheaf
that is the cokernel of a morphism $2\O \to 2\O(1)$, in violation of semi-stability.
Canceling $5\O$ we get the resolution
\[
0 \lra \O(-4) \oplus \O \lra 2\O(1) \lra \E \lra 0.
\]
Combining this with the standard resolution of $\C_x$ tensored with $\O(-2)$ leads us to the
exact sequence
\[
0 \lra \O(-4) \lra \O(-4) \oplus 2\O(-3) \oplus \O \lra \O(-2) \oplus 2\O(1) \lra \F \lra 0.
\]
Since $\Ext^1(\C_x, 2\O(1))=0$, we can argue as at 2.3.2 \cite{mult_five}
to prove that $\F$ would be a split extension of $\C_x$ by $\E$ if the morphism
$\O(-4) \to \O(-4)$ in the above complex were non-zero.
Canceling $\O(-4)$ we get $\f \in W_6$ such that $\F \isom \Coker(\f)$.
\end{proof}

\begin{prop}
\label{8.3}
The sheaves $\F$ giving points in $X_6$ are precisely the non-split extension sheaves of the form
\[
0 \lra \E \lra \F \lra \C_x \lra 0,
\]
where $\E$ gives a point in the stratum $X_6$ of $\M(6,2)$
(cf. 6.1 \cite{mult_six_two}) and $\C_x$ is the structure sheaf of a closed point $x \in \P^2$.
The generic sheaves in $X_6$ have the form $\O_C(2)(P_1-P_2)$, where $C \subset \P^2$
is a smooth sextic curve and $P_1, P_2$ are distinct points on $C$.
In particular, $X_6$ is contained in the closure of $X_5$.
\end{prop}

\begin{proof}
We saw at \ref{8.1} that every $\F$ in $X_6$ is an extension as above.
Conversely, if $\F$ is a non-split extension of $\C_x$ by $\E$, then we can apply the
horseshoe lemma as in the proof of \ref{8.2} to construct $\f \in W_6$ such that $\F \isom \Coker(\f)$.

According to 6.1 \cite{mult_six_two}, generically $\E$ is isomorphic to $\O_C(2)(-P_2)$
for some smooth sextic curve $C \subset \P^2$ and some closed point $P_2 \in C$.
Thus, generically, $\F \isom \O_C(2)(P_1-P_2)$.
Recall from \ref{6.3} that the generic sheaves in $X_5$ have the form
$\O_C(2)(P_1 + Q_1 -P_2 -Q_2)$.
Making $Q_2$ converge to $Q_1$ we produce a sequence of points in $X_5$ converging
to $\F$. Thus $X_6 \subset \overline{X}_5$.
\end{proof}


\section{The smallest stratum}

\begin{prop}
\label{9.1}
The sheaves $\F$ giving pints in $\M(6,3)$ and satisfying the condition $\H^1(\F(1)) \neq 0$
are precisely the sheaves of the form $\O_C(2)$, where $C \subset \P^2$ is a sextic curve.
The set of isomorphism classes of such sheaves, denoted $X_7$, is a closed subvariety of
$\M(6,3)$ that is canonically isomorphic to $\P(\SS^6 V^*)$.
Moreover, $X_7$ is contained in the closure of $X_6$.
\end{prop}

\begin{proof}
Let $\F$ give a point in $\M(6,3)$ and satisfy the condition $\H^1(\F(1))\neq 0$.
The sheaf $\F^{\D}(-1)$ gives a point in $\M(6,-9)$ and has a non-vanishing group of global sections.
Arguing as at 2.1.3 \cite{drezet-maican} we can show that there is an injective morphism
$\O_C \to \F^\D(-1)$ for some curve $C \subset \P^2$ of degree at most $6$.
If $\deg(C) \le 5$, then $\O_C$ would destabilise $\F^\D(-1)$.
Thus $C$ is a sextic curve and $\F^\D(-1) \isom \O_C$, hence $\F \isom \O_C(1)^\D \isom \O_C(2)$.
Conversely, for any sextic curve $C$ the sheaf $\O_C(2)$ is stable and satisfies the
cohomological condition from the proposition.

To prove that $X_7 \subset \overline{X}_6$ make $P_2$ converge to $P_1$ in proposition
\ref{8.3} and note that $[\O_C(2)(P_1 - P_2)]$ converges to $[\O_C(2)]$ if $C$ is smooth.
\end{proof}

\noi
In the remaining part of this section we will prove that $\M(6,3)$ is the union of $X_0, \ldots, X_7$,
i.e. that there are no other semi-stable sheaves on $\P^2$ with Hilbert polynomial $\PP(t)=6t+3$
beside those we have discussed so far.

\begin{prop}
\label{9.2}
Let $\F$ give a point in $\M(6,3)$ and satisfy the condition $\h^0(\F(-1)) \ge 3$ or the condition
$\h^1(\F) \ge 3$. Then $\F \isom \O_C(2)$ for some sextic curve $C \subset \P^2$.
\end{prop}

\begin{proof}
Let $\F$ give a point in $\M(6,3)$ and satisfy the condition $\h^0(\F(-1)) \ge 3$.
Arguing as in 2.1.3 \cite{drezet-maican} we see that there is an injective morphism
$\O_C \to \F(-1)$ for some curve $C \subset \P^2$ of degree at most $6$.
According to remark \ref{3.3} and proposition 5.2, $\F$ is stable.
Thus $\pp(\O_C) < -1/2$, so $C$ has degree $5$ or $6$.
Assume first that $\deg(C)=6$.
The quotient sheaf $\CC = \F/\O_C(1)$ has length $6$ and dimension zero.
Let $\CC' \subset \CC$ be a subsheaf of length $5$ and let $\F'$ be its preimage in $\CC$.
We have an exact sequence
\[
0 \lra \F' \lra \F \lra \C_x \lra 0
\]
in which $\C_x$ is the structure sheaf of a closed point $x \in \P^2$.
We claim that $\F'$ is semi-stable. If this were not the case, then $\F'$ would have a destabilising
subsheaf $\F''$, which may be assumed to be stable.
In fact, $\F''$ must give a point in $\M(5,2)$ because $1/3 < \pp(\F'') < 1/2$.
According to \cite{mult_five}, section 2, we have the inequality $\h^0(\F''(-1)) \le 1$.
The quotient sheaf $\F/\F''$ has Hilbert polynomial $\PP(t)=t+1$ and no zero-dimensional torsion.
Thus $\F/\F'' \isom \O_L$ for some line $L \subset \P^2$.
It follows that
\[
\h^0(\F(-1)) \le \h^0(\F''(-1)) + \h^0(\O_L(-1)) \le 1,
\]
contradicting our hypothesis.
This proves that $\F'$ gives a point in $\M(6,2)$.
We have the relation $\h^0(\F'(-1)) \ge 2$ hence,
according to \cite{mult_six_two}, there is a resolution
\[
0 \lra \O(-4) \oplus \O \lra 2\O(1) \lra \F' \lra 0.
\]
Combining this with the standard resolution of $\C_x$ tensored with $\O(1)$ we get the exact sequence
\[
0 \lra \O(-1) \lra \O(-4) \oplus 3\O \lra 3\O(1) \lra \F \lra 0.
\]
From this we obtain the relation $\h^1(\F(1))=1$, hence,
by \ref{9.1}, $\F$ is isomorphic to $\O_C(2)$.

Assume now that $C$ has degree $5$. The quotient sheaf $\F/\O_C(1)$ has Hilbert polynomial
$\PP(t)=t+3$. If $\F/\O_C(1)$ had zero-dimensional torsion different from zero, then $\F$ would map
surjectively onto the the structure sheaf $\C_x$ of a closed point $x \in \P^2$.
This situation has already been examined.
Thus we may assume that $\F/\O_C(1)$ has no zero-dimensional torsion,
i.e. that $\F/\O_C(1) \isom \O_L(2)$ for some line $L \subset \P^2$.
We apply the horseshoe lemma to the extension
\[
0 \lra \O_C(1) \lra \F \lra \O_L(2) \lra 0,
\]
to the standard resolution of $\O_C(1)$ and to the resolution
\[
0 \lra \O(-1) \lra 3\O \lra 2\O(1) \lra \O_L(2) \lra 0.
\]
We arrive at the resolution
\[
0 \lra \O(-1) \lra \O(-4)  \oplus 3\O \lra 3\O(1) \lra \F \lra 0
\]
leading to the conclusion, as we saw above, that $\F$ is isomorphic to $\O_S(2)$,
where $S = C \cup L$.
\end{proof}

\begin{prop}
\label{9.3}
There are no sheaves $\F$ giving points in $\M(6,3)$ and satisfying the cohomological conditions
\[
\h^0(\F(-1)) \le 1, \qquad \h^1(\F) \ge 2, \qquad \h^1(\F(1)) =0.
\]
\end{prop}

\begin{proof}
The argument can be found at 7.2 \cite{mult_six_one}.
Denote $p= \h^1(\F)$, $m= \h^0(\F \tensor \Om^1(1))$.
Assume that $\F$ gives a point in $\M(6,3)$ and satisfies the conditions $\h^0(\F(-1))=0$,
$\h^1(\F) \ge 2$. The Beilinson free monad for $\F$
\[
0 \lra 3\O(-2) \oplus m\O(-1) \lra m\O(-1) \oplus (p+3) \O \stackrel{\eta}{\lra} p\O \lra 0
\]
leads to a resolution
\[
0 \lra 3\O(-2) \oplus m\O(-1) \stackrel{\f}{\lra} \Ker(\eta_{11}) \oplus (p+3)\O \lra \F \lra 0
\]
in which $\f_{12}=0$. We have the inequality $m-p = \rank(\Ker(\eta_{11})) \le 3$
because $\f$ is injective. Thus
\[
\h^0(\F(1)) = 3(p+3) + \h^0(\Ker(\eta_{11})(1)) - m \ge 2(p+3) \ge 10
\]
forcing $\h^1(\F(1)) > 0$. Assume now that $\h^0(\F(-1)) =1$.
The Beilinson free monad for $\F$ takes the form
\[
0 \lra \O(-2) \stackrel{\psi}{\lra} 4\O(-2) \oplus m\O(-1) \lra
m\O(-1) \oplus (p+3)\O \stackrel{\eta}{\lra} p\O \lra 0.
\]
We have an exact sequence
\[
0 \lra 4\O(-2) \oplus \Coker(\psi_{21}) \stackrel{\f}{\lra} \Ker(\eta_{11}) \oplus (p+3) \O \lra \F \lra 0
\]
in which $\f_{12}=0$. As above, the relation $m-p = \rank(\Ker(\eta_{11})) \le 4$ follows from the
injectivity of $\f$.
If $m-p=4$, then $\Coker(\f_{11})$ would be a destabilising quotient sheaf of $\F$.
Thus $m-p \le 3$ and we conclude as above that $\h^0(\F(1)) > 0$.
\end{proof}


\section{The complement of a codimension $2$ subvariety as a blow-up}

\noi
We saw in section 1 that $\M(6,3)$ is birational to $\N(6,3,3)$, more precisely
the complement of the divisor $\overline{X}_1$ is isomorphic to an open subset
of the Kronecker moduli space. In this section we shall obtain a more detailed picture
of the birational geometry of $\M(6,3)$, namely we shall prove that the complement
of a codimension $2$ subvariety is isomorphic to an open subset of the blow-up
of $\N(6,3,3)$ at the special point represented by the matrix
\[
\left[
\ba{c}
X \\ Y \\Z
\ea
\right] \left[
\ba{ccc}
X & Y & Z
\ea
\right].
\]
Notice that this matrix is stable, so the point it represents, denoted by $s$, lies in the
smooth locus of $\N(6,3,3)$. Let $B$ denote the blow-up of $\N(6,3,3)$ at $s$ and let
$\b \colon B \to \N(6,3,3)$ denote the blowing-down map.
We saw in section 2 that $X_{10}$ is a smooth locally closed subvariety isomorphic
to an open subset of $\P^{36}$.
It is tempting to think of this projective space as the exceptional divisor of $B$.

The main result of this section is that $X_0 \cup X_{10}$ is isomorphic to an open subset
of $B$. The proof will be largely omitted because it is analogous to the proof of \ref{7.2},
only notationally much more cumbersome.
We begin by noting that the union $X = X_0 \cup X_{10}$ is an open subset of $\M(6,3)$
whose complement has two irreducible components, each of codimension $2$,
namely $\overline{X}_{11}$ and $\overline{X}_{11}^\D$.
The inclusion $\overline{X}_{11} \cup \overline{X}_{11}^\D \subset \M(6,3) \setminus X$
follows from the facts that $X_0$ is open and $X_{10}= X_1 \setminus (X_{11} \cup X_{11}^\D)$
is open in $X_1$, cf. 3.1.
The reverse inclusion follows from the fact, proven in section 9, that
\[
\M(6,3) \setminus X = X_{11} \cup X_{11}^\D \cup X_2 \cup X_3 \cup X_3^\D \cup X_4 \cup X_5
\cup X_6 \cup X_7.
\]
Each set on the right is contained in $\overline{X}_{11} \cup \overline{X}_{11}^\D$,
as shown at \ref{4.2}, \ref{4.4}, \ref{5.4}, \ref{6.3}, \ref{8.3}, \ref{9.1}.

In this section we write $G=G_1$. Let $W \subset \W_1$ be the union of $W_{10}$
with the set of morphisms $\f$ equivalent to
\[
\left[
\ba{cc}
0 & 1 \\
\psi & 0
\ea
\right]
\]
for some $\psi \in W_0$. Clearly $W$ is $G$-invariant and $X$ is its image under the map
$\f \mapsto [\Coker(\f)]$.
We claim that $W$ is open in $\W_1$. Indeed, $W$ is the open subset of injective morphisms
having semi-stable cokernel inside the open $G$-invariant subset of $\W_1$ of morphisms
\[
\f = \left[
\ba{cccc}
u_1 & u_2 & u_3 & c \\
\star & \star & \star & v_1 \\
\star & \star & \star & v_2 \\
\star & \star & \star & v_3
\ea
\right]
\]
satisfying the conditions
\[
\spann \{ u_1, u_2, u_3, cX, cY, cZ \} = V^*, \qquad
\spann \{ v_1, v_2, v_3, cX, cY, cZ \} = V^*.
\]
Let $W^\st \subset W$ be the subset of morphisms having stable cokernel
and let $X^\st$ be its image in $\M(6,3)$.
By analogy with proposition \ref{2.2} we have the following:

\begin{prop}
\label{10.1}
The canonical map $W \to X$ is a categorical quotient map for the action of $G$.
The restricted map $W^\st \to X^\st$ is a geometric quotient.
\end{prop}

\noi
Recall from section 2 the subset $\W_0^\ss \subset \W_0$ of morphisms that are
semi-stable as Kronecker $V$-modules and denote by $\g \colon \W_0^\ss \to \N(6,3,3)$
the good quotient map for the action of $G_0$.
Note that $\g^{-1}(\{ s \})$ is the smooth subvariety, denoted by $Z$,
of morphisms represented by matrices of the form
\[
\psi = \left[
\ba{c}
v_1 \\ v_2 \\ v_3
\ea
\right] \left[
\ba{ccc}
u_1 & u_2 & u_3
\ea
\right],
\]
where $\{ u_1, u_2, u_3 \}$ is a basis of $V^*$ and ditto for $\{ v_1, v_2, v_3 \}$.
Following 4.3 \cite{drezet-maican}, we define the morphism
\[
\d \colon W \lra \W_0^\ss, \qquad \d(\f) = c\f_{21} - \f_{22} \f_{11}.
\]
Note that $\g \circ \d$ maps the smooth hypersurface $W_{10}$ to $s$.
By the universal property of the blow-up there is a morphism
$\a \colon W \to B$ making the diagram commute:
\[
\xymatrix
{
W \ar[r]^-{\d} \ar[d]^-{\a} & \W_0^\ss \ar[d]^-{\g} \\
B \ar[r]^-{\b} & \N(6,3,3)
}.
\]
Choose a point $\psi \in Z$ as above.
We identify $\TTT_s \N(6,3,3)$ with the fibre over $\psi$ of the normal bundle of $Z$
in $\N(6,3,3)$. Since $Z$ is smooth, $\TTT_{\psi} Z$ can be identified with the space of
matrices
\[
\left[
\ba{c}
v_1 \\ v_2 \\ v_3
\ea
\right] \left[
\ba{ccc}
u_1' & u_2' & u_3'
\ea
\right] + \left[
\ba{c}
v_1' \\ v_2' \\ v_3'
\ea
\right] \left[
\ba{ccc}
u_1 & u_2 & u_3
\ea
\right], \qquad u_1', u_2', u_3', v_1', v_2', v_3' \in V^*.
\]
Choose $\f \in W_{10}$ lying over $\psi$.
The same calculation as in section 7 shows that
\[
\a(\f) = \langle \f_{21} \mod \TTT_{\psi} Z \rangle \in \P(\TTT_{\psi} \W_0^{\ss}/\TTT_{\psi} Z)
= \P(\TTT_s \N(6,3,3)),
\]
where the r.h.s. is identified with the exceptional divisor of $B$.

\begin{prop}
\label{10.2}
The image of $\a$ is an open subset $B_0$ of $B$ and the map $\a \colon W \to B_0$
is a good quotient for the action of $G$.
Thus $X$ is isomorphic to $B_0$.
\end{prop}

\begin{proof}
Recall the set $W_0 \subset \W_0^\ss$ of proposition \ref{2.3} and notice that
$\d (W \setminus W_{10})= W_0$.
The composite map $W \setminus W_{10} \stackrel{\d}{\to} W_0 \stackrel{\g}{\to} W_0/\!/G_0$
is a good quotient for the action of $G$. Here $W_0/\!/G_0$ is an open subset of
$\N(6,3,3) \setminus \{ s \}$.
Thus $\a(W \setminus W_{10}) = \b^{-1}(W_0/\!/G_0)$ is an open subset of $B$ and
the restriction $\a \colon W \setminus W_{10} \to \b^{-1}(W_0/\!/ G_0)$
is a good quotient for the action of $G$.
Since $W_{10}$ is contained in $W^\st$, the problem reduces to showing that
$\a(W^\st)$ is an open subset $B^\st$ of $B$ and the restriction map $W^\st \to B^\st$
is a good quotient.
We have then $B_0 = B^\st \cup \b^{-1}(W_0/\!/G_0)$.

Arguing as at \ref{7.2} we can show that the fibres of the map $\a \colon W^\st \to B^\st$
are precisely the $G$-orbits and that $\a$ has surjective differential at every point of $W^\st$.
From this we conclude, as at loc.cit., that $B^\st$ is open and that the map
$\a \colon W^\st \to B^\st$ is a geometric quotient modulo $G$.
\end{proof}

\end{document}